\documentclass[12pt]{amsart}
\usepackage{amssymb}

\numberwithin{equation}{section}
\theoremstyle{plain}
\newtheorem{theorem}{Theorem}[section]
\newtheorem{proposition}[theorem]{Proposition}
\newtheorem{cor}[theorem]{Corollary}
\newtheorem{lemma}[theorem]{Lemma}
\theoremstyle{remark}
\newtheorem{rem}[theorem]{Remark}

\theoremstyle{definition}
\newtheorem{defn}{Definition}[section]

\def\R{{\mathbb R}}

\def\p#1{{\left({#1}\right)}}

\def\n#1{{\left\|{#1}\right\|}}
\def\abs#1{{\left|{#1}\right|}}
\def\bra#1{{\langle{#1}\rangle}}

\def\Rnx{{\mathbb R}^n_x}

\def\Rone{{\mathbb R}}
\def\Rn{{{\mathbb R}^n}}
\def\Scal{\mathcal{S}}

\def\Nat{{\mathbb N}}
\def\dslash{d\llap {\raisebox{.9ex}{$\scriptstyle-\!$}}}
\def\jp#1{{\left\langle{#1}\right\rangle}}

\title{Weighted Sobolev
 $L^2$ estimates for a class of Fourier integral operators}

\author[]{Michael Ruzhansky and Mitsuru Sugimoto}
\address{
  Michael Ruzhansky:
  \endgraf
  Department of Mathematics
  \endgraf
  Imperial College London
  \endgraf
  180 Queen's Gate, London SW7 2AZ, UK
  \endgraf
  {\it E-mail address} {\rm m.ruzhansky@imperial.ac.uk}
  \endgraf
  \medskip
  Mitsuru Sugimoto:
  \endgraf
  Department of Mathematics, Graduate School of Science
  \endgraf
  Osaka University
  \endgraf
  Machikaneyama-cho 1-16, Toyonaka, Osaka 560-0043, Japan
  \endgraf
  {\it E-mail address} {\rm sugimoto@math.wani.osaka-u.ac.jp}
  }
\thanks{The first author was supported by the JSPS Invitational
 Research Fellowship.}

\date{\today}
\begin{document}
\begin{abstract}
In this paper we develop elements of the global calculus of
Fourier integral operators in $\Rn$ under minimal decay 
assumptions on phases and amplitudes. We also establish
global weighted Sobolev $L^2$ estimates for a class
of Fourier integral operators that appears in the
analysis of global smoothing problems for dispersive
partial differential equations. As an application,
we exhibit a new type of smoothing estimates for
hyperbolic equations, where the decay of data in space
is quantitatively translated into the time decay of
solutions. 
\end{abstract}

\maketitle

\section{Introduction}

The paper will deal with a class of Fourier integral operators
on $\Rn$ that appears in the analysis of smoothing problems
for dispersive evolution equations, such as hyperbolic,
Schr\"odinger, KdV, and others. 

Apart from this motivation, we are interested in global
properties of the calculus of pseudo-differential and Fourier
integral operators under minimal assumptions. This follows
a general line of research for deriving regularity 
properties of pseudo-differential operators under minimal
assumptions on the symbols. Global $L^2$ boundedness properties
of pseudo-differential operators have been thoroughly 
investigated in the literature (e.g. 
Calder\'on and Vaillancourt \cite{CV}, Childs \cite{Ch},
Coifman and Meyer \cite{CM}, Cordes \cite{Co}, 
Sugimoto \cite{Su}, etc.).
These results have been later extended to classes of
Fourier integral operators by the authors in \cite{RS}. 
Local $L^p$ and other estimates for non-degenerate
Fourier integral operators are also quite well understood,
see e.g. the survey paper by the first author \cite{R}.

Results on the global $L^2$ boundedness of Fourier integral 
operators in $\Rn$ appeared in the papers of
Asada and Fujiwara \cite{AF}, Kumano-go \cite{Ku},
Boulkhemair \cite{Bo}, with different assumptions on the
regularity of phases and amplitudes. In all these papers
the assumption was made that the second order derivatives of
the phase with respect to the frequency variables must be
globally bounded. Unfortunately, this assumption fails in
applications to global smoothing problems, and it was 
recently removed by the authors in \cite{RS4}.

The method of dealing with
dispersive equations by globally reducing them to normal
forms has been developed by the authors in 
\cite{RS2}. This method proved to be very effective 
for a variety of equations. For example, a critical case of
the Kato--Yajima \cite{KY}
smoothing estimate for Schr\"odinger equations was proved
in \cite{RS3} by globally reducing the estimate to a model
case with useful additional
commutator properties. The same strategy
was later developed in \cite{RS6} to treat smoothing estimates
for general dispersive equations with constant 
and time-dependent coefficients.
Instead of previously used spectral methods or dual restriction
theorems from harmonic analysis, the approach of
\cite{RS2}--\cite{RS5} relies on global $L^2$ estimates for
the canonical transforms that in turn can be realised as
a class of globally defined Fourier integral operators in $\Rn$. 
Moreover, smoothing estimates themselves are essentially global 
weighted Sobolev $L^2$ estimates for solutions to dispersive
equations, and together with Strichartz estimates they
constitute the main tools for dealing with time well-posedness
questions for nonlinear evolution equations.

Transforming smoothing estimates by the method
of canonical transforms leads to the necessity to analyse
weighted Sobolev estimates for Fourier integral operators.
In \cite{RS4}, the authors investigated questions of $L^2$ and
weighted $L^2$ boundedness for the appearing operators.
In Theorem \ref{th:sob1} of this paper we will now give a result in 
weighted Sobolev spaces over $L^2$. This result will rely on
a global calculus which is developed in Section 2.
Rudiments of such estimates without asymptotic formulae
but with estimates for remainders appeared in \cite{RS3},
while elements of the calculus (without proofs) were 
announced in \cite{RS5}. In this sense, the present paper
is the continuation and the development of the subjects
of authors' papers \cite{RS3} and \cite{RS4}. In particular,
the first part of this paper will provide a calculus 
background to the better understanding
of estimates of the remainders in \cite{RS3}, this allowing
further applications to similar problems with certain 
structures in operators.

Thus, the first
aim of this paper is to develop the calculus necessary
for these purposes, at the same time investigating the independent
question of what are the minimal decay assumptions for
the global calculus to exist. Consequently, we are also interested
in global weighted estimates for Fourier integral operators
in Sobolev spaces, and in
applications to smoothing problems. This is especially 
important in the investigation of global smoothing properties
for equations with variable coefficients, a topic that will
be treated in detail elsewhere.

We will be dealing with operators $T$ that can
be globally defined by
\begin{equation}\label{eq:tdef}
Tu(x)=\int_\Rn\int_\Rn e^{i(\phi(x,\xi)-y\cdot\xi)} a(x,y,\xi)u(y)
dy \dslash\xi,
\end{equation}
with a real-valued phase function $\phi$ and a 
smooth amplitude $a(x,y,\xi)$, satisfying estimates
\[
|\partial_x^\alpha\partial_y^\beta\partial_\xi^\gamma
a(x,y,\xi)|\leq
C_{\alpha\beta\gamma}\bra{x}^{m_1}\bra{y}^{m_2}\bra{\xi}^{m_3},
 \textrm{ for all }x,y,\xi\in\Rn,
\]
for all multi-indices $\alpha, \beta, \gamma$. 
Here and in the sequel we use the notation
$\dslash\xi=(2\pi)^{-n}d\xi$ and
$\bra{x}=(1+|x|^2)^{1/2}.$ 
As usual, we understand these
operators as limits in $\Scal(\Rn)$ as $\epsilon\to 0$, of
operators with amplitudes
$a_\epsilon(x,y,\xi)=a(x,y,\xi)\gamma(\epsilon y,\epsilon \xi)$,
where $\gamma\in C_0^\infty(\Rn\times\Rn)$ equals 1 near the
origin. These operators can be naturally 
extended to $\Scal^\prime(\Rn).$ We will adopt the minus sign in 
the exponent for some technical convenience,
and also in order to make the exposition directly applicable to
pseudo-differential operators, thus providing new weighted
Sobolev estimates for pseudo-differential operators as well.

Let us note that operators with amplitude $a(x,y,\xi)=a(x,\xi)$
independent of $y$, and with
representation (\ref{eq:tdef}) locally in
$x$, appear as solutions to the Cauchy problem for strictly
hyperbolic equations. Their properties have been studied
extensively (see, e.g. Kumano-go \cite{Ku1}, and also 
Stein \cite[Chapter IX]{St}). A global in $x$
representation with amplitude satisfying some decay has been 
also used
to study global properties of solutions to the Cauchy problem.
Thus, problems with phases and amplitudes in 
so-called SG--classes were investigated
by Coriasco \cite{Cor}. In \cite{BBR}, Boggiatto, Buzano and
Rodino investigated
global problems with phases and amplitudes from some
polynomial classes associated to the Newton polygon of the
hyperbolic operator. In comparison with these results, we will
impose much less restrictive 
conditions on phases and amplitudes. In addition, from our
theorems on compositions, it will be 
easy to see that when the original
amplitude satisfies some additional decay conditions, the
amplitude of the composition will satisfy the same conditions 
too. In this way we will automatically recover global calculi
developed for SG--operators by Coriasco \cite{Cor},
for operators with certain polynomial weights by
Boggiatto, Buzano, and Rodino \cite{BBR}, and to a
certain extent for
operators of Gelfand--Shilov type by Cappiello \cite{Ca},
as special cases of our calculus.
Below we will give an example of such an argument.

A special case of operator $T$ occurs when
$\phi(x,\xi)=x\cdot\xi.$ In this case $T$ is a pseudo-differential
operator. In Corollary \ref{cor:repr1} we will show that if
derivatives of the amplitude $a(x,y,\xi)$
decay in the second variable, $T$ can
be defined by an amplitude $p(x,\xi)$ of order $m_1+m_2$ in $x$,
and $m_3$ in $\xi$, in the usual form of pseudo-differential 
operators.

To model weighted Sobolev spaces, we will use a convenient class
of operators, the so-called SG pseudo-differential operators.
We will say that $p=p(x,\xi)\in SG^{t_1,t_2}(\Rn\times\Rn)$ if $p$
is smooth and if for all $\alpha, \beta$ we have the estimate
\[
|\partial_\xi^\alpha\partial_x^\beta p(x,\xi)|\leq C_{\alpha\beta}
\bra{x}^{t_1-|\beta|} \bra{\xi}^{t_2-|\alpha|},
\textrm{ for all } x,\xi\in\Rn.
\]
The corresponding pseudo-differential operator 
$P\in Op(SG^{t_1,t_2})$ is
defined by
\begin{equation}\label{eq:pdef}
Pu(x)=\int_\Rn\int_\Rn e^{i(x-y)\cdot\xi} p(x,\xi)u(y) dy 
\dslash\xi.
\end{equation}
Note that in principle one can consider more general SG--amplitudes
$a(x,y,\xi)$ instead of $p(x,\xi)$ in
(\ref{eq:pdef}), i.e. smooth functions $a$,
for some $m_1, m_2, m_3\in\Rone$
satisfying
\[
|\partial_\xi^\alpha\partial_x^\beta\partial_y^\gamma
a(x,y,\xi)|\leq C_{\alpha\beta\gamma} \bra{x}^{s_1-|\beta|}
\bra{y}^{s_2-|\gamma|}\bra{\xi}^{s_3-|\alpha|},
\textrm{ for all } x,y,\xi\in\Rn,
\]
and all multi-indices $\alpha, \beta, \gamma$.
In \cite{Co2}, Cordes shows that such symbols can be reduced to
$p=p(x,\xi)\in SG^{s_1+s_2,s_3}$ with an asymptotic formula
\[
p(x,\xi)\thicksim\sum_{\alpha\in\Nat^n}
\frac{i^{-|\alpha|}}{\alpha !} \partial_\xi^\alpha
\partial_y^\alpha a(x,y,\xi)|_{y=x}.
\]
Let us show that we automatically recover such 
(and other similar) results in the
framework of our calculus. Indeed,
from the asymptotic formula of Corollary \ref{cor:repr} we
see that if the amplitude $a(x,y,\xi)$ is an SG--amplitude, then
$p(x,\xi)$ must be an SG--amplitude as well. This type of argument
can be used with different decay assumptions (such as
polynomial, quasi-homogeneous, 
Gelfand--Shilov, etc.), to yield similar
conclusions.

We note also that the adjoint operator $T^*$ of $T$ 
from \eqref{eq:tdef} is given by
\[
T^*v(x)=\int_\Rn\int_\Rn
e^{i(x\cdot\xi-\phi(y,\xi))}\overline{a(y,x,\xi)} v(y) dy 
\dslash\xi,
\]
and it is
operators of this type that appear in global smoothing problems
(\cite{RS2}--\cite{RS4}).

In Section 2 we will develop elements of the global calculus
dealing with global compositions of
operators \eqref{eq:tdef} with pseudo-differential operators.
In the proofs of Theorem \ref{prop:alt} and Lemma
\ref{l:composition}, we will assume that all integrations by parts
are carried out even number of times, so that we do not have 
to keep track of the sign in front of the integral.
Since we make a choice to represent the phase not symmetrically
with respect to $x$ and $y$, we will have non-symmetric
statements also
in our theorems. Thus, we will have two different versions of
composition formulae for $T\circ P$ based on whether we prefer
to keep an arbitrary phase (as in Theorem \ref{prop:alt})
or to have the amplitude dependent on less variables
(as in Theorem \ref{th:tp}). For the composition $P\circ T$
pseudo-differential operator $P$ interacts with the phase
of $T$, and we will have Theorem \ref{th:pt}
dealing with this case.
In terms of the growth properties of the phase function,
we will usually think of a typical situation for
Fourier integral operators when the phase
$\phi$ is of order one in both $x$ and $\xi$ and the growth
may be cancelled once we take a derivative (for precise
assumptions see the corresponding theorems).

In Section 3 we will apply the established
calculus to derive weighted
Sobolev $L^2$ estimates for our classes of 
Fourier integral operators. This will extend
weighted $L^2$ estimates established by the authors
in \cite{RS4} to weighted Sobolev spaces over 
$L^2$, the setting important
for smoothing problems.

In Section 4 we will apply the calculus and
weighted estimates to establish a number of new smoothing
estimates for hyperbolic equations, relating time and space
decay of solution to that of the Cauchy data. Smoothing estimates
for dispersive equations have been thoroughly analysed in
the literature (e.g. \cite{BD}, \cite{BK}, \cite{CS},
\cite{KY}, \cite{RS6},
\cite{Si}, \cite{Sj}, \cite{V}, \cite{Wa}, and many other papers). 
The idea is that integration
of the solution to an evolution equation with respect to time
brings an improvement in the regularity with respect to $x$.
This is a phenomenon similar to the one exhibited in Strichartz
estimates where one observes an improved behaviour in the scale
of $L^p$ spaces. Smoothing estimates are centred
at the improvement in Sobolev
regularity in $L^2$, and this improvement
depends on the order of the evolution equation. 
Moreover, for applications to the time well-posedness 
problems for nonlinear
equations, one is looking for global smoothing estimates,
in which one inserts additional weights to make estimates
work globally in space--time. In this respect,
the class of hyperbolic equations
can be viewed as the endpoint case for smoothing estimates,
in the sense that there is no improvement in regularity, but 
a question of the weight remains. In fact, the problem of
finding the weight can be actually 
viewed as the most important here,
since the method developed by the authors in \cite{RS6} allows
to extend such estimates further to other evolution equations
(e.g. to Schr\"odinger, KdV, and others), automatically
gaining the improvement in the regularity with respect
to $x$. Estimates presented
in Theorem \ref{ths} are a new type of estimates for hyperbolic
equations that we can derive using the global calculus developed 
in the first part of the paper.

If we do not specify the
domain of integration, it will always mean that the
integration is performed over the whole space $\Rn$.
As usual, we will denote $D^\alpha=i^{-|\alpha|}\partial_\alpha$.
We also denote
$\dslash\xi=(2\pi)^{-n}d\xi$ and $\bra{x}=(1+|x|^2)^{1/2}.$
Constants will be usually denoted by letter $C$, and their
values may differ on different occasions.

\section{Composition with pseudo-differential operators}

In this section we will
study the composition of $T$ with pseudo-differential
operators. Let $T$ be given by
\eqref{eq:tdef} and let $P$ be given by \eqref{eq:pdef}.
The composition $PT$ is then of the form
\begin{equation}\label{eq:pt}
\begin{aligned}
(PTu)(x) &
= \int_\Rn\int_\Rn 
e^{i(x-y)\cdot\eta} p(x,\eta) (Tu)(y) dy \dslash\eta \\
& = \int_\Rn\int_\Rn\int_\Rn\int_\Rn
 e^{i((x-y)\cdot\eta+\phi(y,\xi)-z\cdot\xi)}
p(x,\eta) a(y,z,\xi) u(z) dz \dslash\xi dy \dslash\eta \\
& = \int_\Rn\int_\Rn
 e^{i(\phi(x,\xi)-z\cdot\xi)} c_1(x,z,\xi) u(z) dz 
\dslash\xi,
\end{aligned}
\end{equation}
where
\begin{equation}\label{eq:c1}
c_1(x,z,\xi)=\int_\Rn\int_\Rn
 e^{i(\phi(y,\xi)-\phi(x,\xi)+(x-y)\cdot\eta)}
a(y,z,\xi) p(x,\eta) dy \dslash\eta.
\end{equation}
The composition $TP$ is of the form
\begin{equation}\label{eq:tp}
\begin{aligned}
(TPu)(x) &
= \int_\Rn\int_\Rn 
e^{i(\phi(x,\xi)-y\cdot\xi)} a(x,y,\xi) (Pu)(y) dy 
\dslash\xi \\
& = \int_\Rn\int_\Rn\int_\Rn\int_\Rn
 e^{i(\phi(x,\xi)-y\cdot\xi+(y-z)\cdot\eta)}
a(x,y,\xi) p(y,\eta) u(z) dz \dslash\eta dy \dslash\xi \\
& = \int_\Rn\int_\Rn 
e^{i(\phi(x,\eta)-z\cdot\eta)} c_2(x,z,\eta) u(z) dz
\dslash\eta,
\end{aligned}
\end{equation}
where
\begin{equation}\label{eq:c2}
c_2(x,z,\eta)= c_2(x,\eta)=\int_\Rn\int_\Rn
e^{i(\phi(x,\xi)-\phi(x,\eta)+y\cdot(\eta-\xi))} a(x,y,\xi)
p(y,\eta) dy \dslash\xi.
\end{equation}
Note that this representation of $TP$ makes the amplitude
dependent on two variables only
(independent of $z$). This will be used later in
Corollary \ref{cor:repr}. We also note that both amplitudes $c_1$
and $c_2$ are of the same type, if we interchange variables $x, y$
with variables $\xi, \eta$, respectively. The difference is that
the integration is carried out
with respect to one or two variables of the
amplitude $a$ in formulae for
$c_1$ and $c_2$, respectively. We will use this fact
and will treat both
cases simultaneously in Lemma \ref{l:composition}.

Let us now use another representation of $TP$ to obtain a slightly
different estimate for its amplitude. 
In fact, one can readily see that we can represent $TP$ as
\begin{equation*}\label{eq:tpalt}
(TPu)(x)=\int_\Rn\int_\Rn e^{i(\phi(x,\xi)-z\cdot\xi)}
c(x,z,\xi) u(z) dz \dslash\xi,
\end{equation*}
with amplitude
\begin{equation}\label{eq:calt}
c(x,z,\xi)=\int_\Rn\int_\Rn e^{i(y-z)\cdot(\eta-\xi)}
a(x,y,\xi) p(y,\eta) dy \dslash\eta.
\end{equation}
Note that in comparison with representation 
(\ref{eq:tp})--(\ref{eq:c2}), 
the amplitude here depends on three variables, but
there is no entry of the phase. This allows to treat more general
phases and amplitudes, leading to the following:

\begin{theorem}[Composition $TP$]\label{prop:alt}
Let operator $T$ be defined by
\[
Tu(x)=\int_\Rn\int_\Rn e^{i(\phi(x,\xi)-y\cdot\xi)} a(x,y,\xi)
u(y)\,dy\,\dslash\xi.
\]
Let the phase $\phi=\phi(x,\xi)$ be any function and
$a=a(x,y,\xi)\in 
C^\infty(\Rone^{n}\times\Rn\times\Rn)$ satisfy
\begin{equation}\label{est:amp0}
|\partial_x^\alpha\partial_y^\beta\partial_\xi^\gamma
a(x,y,\xi)|\leq
C_{\alpha\beta\gamma}\bra{x}^{m_1}\bra{y}^{m_2}\bra{\xi}^{m_3},
\textrm{ for all } x,y,\xi\in\Rn,
\end{equation}
for all $\alpha, \beta, \gamma$. 
Let $p=p(y,\eta)\in C^\infty(\Rn\times\Rn)$ for all $\alpha,
\beta$ satisfy the estimate
\begin{equation}\label{est:p0}
|\partial_y^\alpha\partial_\eta^\beta p(y,\eta)|\leq
C_{\alpha\beta}\bra{y}^{t_1} \bra{\eta}^{t_2-|\beta|},
\textrm{ for all } y,\eta\in\Rn.
\end{equation}
Then the composition $B=T\circ P(x,D)$ is an operator of the form
\[
(TPu)(x)=\int_\Rn\int_\Rn e^{i(\phi(x,\xi)-z\cdot\xi)}
c(x,z,\xi) u(z) dz \dslash\xi,
\]
with amplitude $c(x,z,\xi)$ satisfying
\begin{equation}\label{est:comp0}
|\partial_x^\alpha\partial_z^\beta\partial_\xi^\gamma c(x,z,\xi)|
\leq C_{\alpha\beta\gamma}\bra{x}^{m_1}\bra{z}^{m_2+t_1}
\bra{\xi}^{m_3+t_2},
\end{equation}
for all $\alpha,\beta,\gamma$ and all $x,z,\xi\in\Rn$
Moreover, we have the asymptotic expansion, improving in $\xi$:
\[
c(x,z,\xi) \thicksim \sum_\alpha
\frac{i^{-|\alpha|}}{\alpha!}\partial_y^\alpha\left[ a(x,y,\xi)
\partial_\xi^\alpha p(y,\xi)\right]|_{y=z}.
\]
\end{theorem}

\begin{defn}[{\bf Improving asymptotic expansions}]
In this paper, when we say that the asymptotic expansion
$a\thicksim \sum_{j=1}^\infty a_j$ is improving in $\xi$, 
it will always mean that
for every $\alpha, \beta, \gamma$ and $M$ there 
exist $N$ and $C$ such that
\begin{equation*}\label{eq:asympt0}
\left|\partial_x^\alpha\partial_y^\beta\partial_\xi^\gamma\left(
a-\sum_{j=1}^N a_j\right)\right|\leq C\jp{x}^{m_1}\jp{y}^{m_2}
\jp{\xi}^{m_3-M}, \textrm{ for all } x,y,\xi\in\Rn.
\end{equation*}
Similarly, we can define expansions improving in $x$ or in $y$.
\end{defn}

\begin{rem}\label{rem1}
It is clear that if $a$ has additional decay with respect to some variables,
so does the new amplitude.
For example, if we assume stronger estimate
\begin{equation*}\label{est:amp1}
|\partial_x^\alpha\partial_y^\beta\partial_\xi^\gamma
a(x,y,\xi)|\leq
C_{\alpha\beta\gamma}\bra{x}^{m_1-|\alpha|}\bra{y}^{m_2}\bra{\xi}^{m_3},
\textrm{ for all } x,y,\xi\in\Rn
\end{equation*}
instead of \eqref{est:amp0}, then
we can improve estimate \eqref{est:comp0} as
\begin{equation}\label{est:comp1}
|\partial_x^\alpha\partial_z^\beta\partial_\xi^\gamma c(x,z,\xi)|
\leq C_{\alpha\beta\gamma}\bra{x}^{m_1-|\alpha|}\bra{z}^{m_2+t_1}
\bra{\xi}^{m_3+t_2}.
\end{equation}
Similarly, if we assume
\begin{equation}\label{est:amp2}
|\partial_x^\alpha\partial_y^\beta\partial_\xi^\gamma
a(x,y,\xi)|\leq
C_{\alpha\beta\gamma}\bra{x}^{m_1}\bra{y}^{m_2-|\beta|}\bra{\xi}^{m_3},
\textrm{ for all } x,y,\xi\in\Rn
\end{equation}
and
\begin{equation}\label{est:p2}
|\partial_y^\alpha\partial_\eta^\beta p(y,\eta)|\leq
C_{\alpha\beta}\bra{y}^{t_1-|\alpha|} \bra{\eta}^{t_2-|\beta|}
\textrm{ for all } y,\eta\in\Rn
\end{equation}
instead of estimates \eqref{est:amp0} and \eqref{est:p0}, then
we can improve estimate \eqref{est:comp0} as
\begin{equation*}\label{est:comp2}
|\partial_x^\alpha\partial_z^\beta\partial_\xi^\gamma c(x,z,\xi)|
\leq C_{\alpha\beta\gamma}\bra{x}^{m_1}\bra{z}^{m_2+t_1-|\beta|}
\bra{\xi}^{m_3+t_2}.
\end{equation*}
In fact, from the expression \eqref{eq:calt}, we obtain
\[
\jp{x}^{\alpha}\partial_x^\alpha c(x,z,\xi)
=\int_\Rn\int_\Rn e^{i(y-z)\cdot(\eta-\xi)}
\jp{x}^{\alpha}(\partial_x^\alpha a)(x,y,\xi) p(y,\eta) dy \dslash\eta,
\]
and the conclusion \eqref{est:comp1} follows from \eqref{est:comp0}
if we regard
$\jp{x}^{\alpha}(\partial_x^\alpha a)(x,y,\xi)$
as $a(x,y,\xi)$ in \eqref{est:amp0}.
The same argument is true for the case of assumptions
\eqref{est:amp2} and \eqref{est:p2} since we obtain
\begin{align*}
\partial_{z_j}c(x,z,\xi)
=&\int_\Rn\int_\Rn e^{i(y-z)\cdot(\eta-\xi)}
(\partial_{y_j}a)(x,y,\xi) p(y,\eta) dy \dslash\eta
\\
+&
\int_\Rn\int_\Rn e^{i(y-z)\cdot(\eta-\xi)}
a(x,y,\xi) (\partial_{y_j}p)(y,\eta) dy \dslash\eta,
\\
z_jc(x,z,\xi)
=&\int_\Rn\int_\Rn e^{i(y-z)\cdot(\eta-\xi)}
y_ja(x,y,\xi) p(y,\eta) dy \dslash\eta
\\
+&
\int_\Rn\int_\Rn e^{i(y-z)\cdot(\eta-\xi)}
a(x,y,\xi) (\partial_{\eta_j}p)(y,\eta) dy \dslash\eta,
\end{align*}
from the expression \eqref{eq:calt} again.
From this we can conclude estimate \eqref{est:comp1} 
similar to the argument before.
\end{rem}

\begin{proof}[Proof of Theorem \ref{prop:alt}]
Let $\chi=\chi(y,z)\in C^\infty(\Rone^{n}\times\Rn)$ be such that
$\chi(y,z)=1$ for $|y-z|< k/2$, and $\chi(y,z)=0$ for $|y-z|>k,$
for some small $0<k<1$ to be chosen later.
Using this cut-off function, we decompose $c=c^I+c^{II}$, where
$c^I$ and $c^{II}$ are of the form (\ref{eq:calt}), but
with amplitude $a$ replaced by amplitudes
$a^I=(1-\chi)a$ and $a^{II}=\chi a$, respectively.

{1. Estimate for $|y-z|\geq k.$} We integrate by parts with transposes
of operators $^tL_\eta=-|y-z|^{-2}\Delta_\eta$ and
$^tL_y=\bra{\xi}^{-2}(1-\Delta_y)$, for which we have
$$(^tL_\eta)^{N_1}e^{i(y-z)\cdot(\eta-\xi)}=e^{i(y-z)
\cdot(\eta-\xi)},
\quad (^tL_y)^{N_1}e^{-i(y-z)\cdot\xi}=e^{-i(y-z)\cdot\xi}.$$ 
Thus, we obtain
\[
\begin{aligned}
c^I(x,z,\xi) & = \iint e^{-i(y-z)\cdot\xi} L_y^{N_2}
 \left[
 e^{i(y-z)\cdot\eta} a^I(x,y,\xi)
 L_\eta^{N_1}p(y,\eta)\right] dy \dslash\eta \\
& =  \iint e^{i(y-z)\cdot(\eta-\xi)} e^{-i(y-z)\cdot\eta}
L_y^{N_2} \left[
 e^{i(y-z)\cdot\eta} a^I(x,y,\xi)
 L_\eta^{N_1}p(y,\eta)\right] dy \dslash\eta \\
& =  \iint e^{-iz\cdot(\eta-\xi)}
\frac{(1-\Delta_\eta)^{N_3}}{\bra{z}^{2N_3}}\times \\
& \quad\quad \times \left[
e^{iy\cdot(\eta-\xi)} e^{-i(y-z)\cdot\eta} L_y^{N_2} \left[
 e^{i(y-z)\cdot\eta} a^I(x,y,\xi)
 L_\eta^{N_1}p(y,\eta)\right]\right] dy \dslash\eta.
\end{aligned}
\]
Now, if we use that after all differentiations,
$e^{i(y-z)\cdot\eta}$ in the last integral
cancels, we get
\[
\begin{aligned}
|c^I(x,z,\xi)| & \leq C\iint_{|y-z|>k} \bra{z}^{-2N_3}
\bra{y}^{2N_3}\bra{\xi}^{-2N_2}
 \bra{\eta}^{2N_2}
\bra{x}^{m_1} \bra{y}^{m_2} \bra{\xi}^{m_3}\times \\
 & \quad \quad \times |y-z|^{-2N_1} \bra{y}^{t_1}
\bra{\eta}^{t_2-2N_1} dy d\eta \\ &
\leq C \bra{x}^{m_1} \bra{z}^{-2N_3} \bra{\xi}^{m_3-2N_2}
\int_{|y-z|>k} \frac{\bra{y}^{2N_3+m_2+t_1}}{|y-z|^{2N_1}} dy
\int \bra{\eta}^{2N_2+t_2-2N_1} d\eta \\
& \leq C \bra{x}^{m_1} \bra{z}^{-2N_3} \bra{\xi}^{m_3-2N_2}
\int_{|y|>k} \frac{\bra{y+z}^{2N_3+m_2+t_1}}{|y|^{2N_1}} dy \\
&  \leq C \bra{x}^{m_1} \bra{z}^{m_2+t_1} \bra{\xi}^{m_3-2N_2}
\int_{|y|>k} \bra{y}^{2N_3+m_2+t_1-2N_1} dy, \\
\end{aligned}
\]
where we have used the estimate $\bra{y+z}\leq C\bra{y}\bra{z}$
for $N_3$ such that $2N_3+m_2+t_1>0$. This gives the desired
estimate for $c^I$ for $N_1\gg N_3, N_2$, and $N_2>m_3$.

For the estimate of derivatives of $c^I(x,z,\xi)$ we first note
that derivatives with respect to $x$ do not have any effect on our
estimates. Derivatives with respect to $z$ introduce the factor of
$\eta-\xi$, which is taken care of by choosing large $N_1, N_2$.
Derivatives with respect to $\xi$ may introduce the factor of
$y-z$, which is taken care of by taking large $N_1$.

{2. Estimate for $|y-z|<k.$} We set $\eta=\xi+\theta$, so the
Taylor expansion of $p$ with respect to $\eta$ gives
\[
\begin{aligned}
p(y,\xi+\theta) & = \sum_{|\alpha|<M} \frac{\theta^\alpha}{\alpha
!}
\partial_\xi^\alpha p(y,\xi) + \sum_{|\alpha|=M}
C_\alpha {\theta^\alpha} r_\alpha(y,\xi,\theta), \\
r_\alpha(y,\xi,\theta) & = \int_0^1 (1-t)^M
\partial_\xi^{\alpha}p(y,\xi+t\theta) dt.
\end{aligned}
\]
Then
\[
\begin{aligned}
c^{II}(x,z,\xi) & = \iint e^{(y-z)\cdot\theta} \chi(y,z)
a(x,y,\xi)
p(y,\xi+\theta) dy \dslash\theta \\
& = \sum_{|\alpha|<M} \frac{1}{\alpha !} c_\alpha(x,z,\xi) +
 \sum_{|\alpha|=M}
 C_\alpha R_\alpha(x,z,\xi),
\end{aligned}
\]
where
\[
\begin{aligned}
R_\alpha(x,z,\xi) & = \iint e^{i(y-z)\cdot\theta}
\theta^\alpha \chi(y,z) a(x,y,\xi) r_\alpha(y,\xi,\theta) dy 
\dslash\theta, \\
c_\alpha(x,z,\xi) & = \iint e^{(y-z)\cdot\theta}\theta^\alpha
\left[ \chi(y,z) a(x,y,\xi) \partial_\xi^\alpha p(y,\xi) \right] 
dy \dslash\theta \\
& = D_y^\alpha\left[ \chi(y,z) a(x,y,\xi)
\partial_\xi^\alpha p(y,\xi)\right]|_{y=z}.
\end{aligned}
\]
From this we readily obtain
\[
|c_\alpha(x,z,\xi)| \leq \bra{x}^{m_1} \bra{z}^{m_2+t_1}
\bra{\xi}^{m_3+t_2-|\alpha|}.
\]
It is easy to see that derivatives of $c_\alpha$ satisfy the same
estimates. For the asymptotic formula we note that
$\partial_y^\alpha\chi(y,z)|_{y=z}=0$ for all $\alpha\not=0$,
so all derivatives must fall on the product
$a\, \partial_\xi^\alpha p$.
Finally, it is easy to see that derivatives of $c_\alpha$ satisfy
similar estimates as $c_\alpha$.

{3. Estimate for the remainder $R_\alpha$.} Let $\rho\in
C_0^\infty(\Rn)$ be such that $0\leq\rho\leq 1$, $\rho(x)=1$ for
$|x|<\epsilon/2$ and $\rho(x)=0$ for $|x|>\epsilon$,
for some $\epsilon>0$ to be chosen later. We decompose
\[
\begin{aligned}
R_\alpha(x,z,\xi) & = R_\alpha^I(x,z,\xi)+R_\alpha^{II}(x,z,\xi) \\
 & = \iint e^{i(y-z)\cdot\theta}
\theta^\alpha \rho\left(\frac{\theta}{\bra{\xi}}
\right) \chi(y,z) a(x,y,\xi)
r_\alpha(y,\xi,\theta) dy \dslash\theta \\ & + \iint
e^{i(y-z)\cdot\theta} \theta^\alpha
\left(1-\rho\left(\frac{\theta}{\bra{\xi}}\right)\right) 
\chi(y,z) a(x,y,\xi)
r_\alpha(y,\xi,\theta) dy \dslash\theta.
\end{aligned}
\]

{3.1. Estimate for $|\theta|\leq \epsilon\bra{\xi}.$} For
sufficiently small $\epsilon>0$, for any $0\leq t\leq 1$,
$\bra{\xi+t\theta}$ and $\bra{\xi}$ are equivalent, which means
that $|\partial_y^\beta r_\alpha(y,\xi,\theta|\leq C\bra{y}^{t_1}
\bra{\xi}^{t_2-|\alpha|}$, for $|\theta|\leq\epsilon\bra{\xi}$,
and all $\alpha, \beta$.
We have
\[
\begin{aligned}
R_\alpha^I(x,z,\xi) & = \iint e^{i(y-z)\cdot\theta}
D_y^\alpha\left[ \rho\left(\frac{\theta}{\bra{\xi}}\right) \chi(y,z)
a(x,y,\xi) r_\alpha(y,\xi,\theta) \right] dy \dslash\theta, \\
|R_\alpha^I(x,z,\xi)| & \leq C \bra{x}^{m_1}
\bra{\xi}^{m_3+t_2-|\alpha|} \int_{|\theta|<\epsilon\bra{\xi}}
d\theta \int_{|y-z|<k} \bra{y}^{m_2+t_1} dy \\ & \leq C
\bra{x}^{m_1} \bra{\xi}^{m_3+t_2-|\alpha|+n} \bra{z}^{m_2+t_1},
\end{aligned}
\]
which yields estimate (\ref{eq:cest}) for $M=|\alpha|\geq n$ and
large $z$.

{3.2. Estimate for $|\theta|> \epsilon\bra{\xi}.$} Here we will
use that fact that for $M=|\alpha|\geq t_2,$ the function
$r_\alpha(y,\xi,\theta)$ and its derivatives can be estimated by
$\bra{y}^{t_1}.$ We will integrate by parts with operator
$L_y=-|\theta|^{-2}\Delta_y,$ for which we have $L^N
e^{i(y-z)\cdot\theta}=e^{i(y-z)\cdot\theta}.$ We have
\[
\begin{aligned}
R_\alpha^{II}(x,z,\xi) & = \iint e^{i(y-z)\cdot\theta}
D_y^\alpha\left[ \left(1-\rho\left(\frac{\theta}{\bra{\xi}}
\right)\right) \chi(y,z)
a(x,y,\xi) r_\alpha(y,\xi,\theta)
\right] dy \dslash\theta \\
& = \iint e^{i(y-z)\cdot\theta} |\theta|^{-2N} \Delta_y^N
D_y^\alpha \\
& \qquad \left[ \left(1-\rho\left(\frac{\theta}{\bra{\xi}}
\right)\right) \chi(y,z)
a(x,y,\xi) r_\alpha(y,\xi,\theta) \right] dy \dslash\theta,
\end{aligned}
\]
which implies
\[
\begin{aligned}
 |R_\alpha^{II}(x,z,\xi)| & \leq C
\bra{x}^{m_1} \bra{\xi}^{m_3} \int_{|\theta|>\epsilon\bra{\xi}}
|\theta|^{-2N} d\theta \int_{|y-z|<k} \bra{y}^{m_2+t_1} dy \\ &
\leq C \bra{x}^{m_1} \bra{\xi}^{m_3+n-2N} \bra{z}^{m_2+t_1}.
\end{aligned}
\]
Derivatives of $R_\alpha^I, R_\alpha^{II}$ with respect to $z$ and
$\xi$ may introduce additional powers of $\theta$, which amounts
to taking more $y$-derivatives under the integral and does not
change the estimates. This completes the proof of Theorem
\ref{prop:alt}.
\end{proof}

Now we will prove a lemma which will be used in deriving estimates
for amplitudes in compositions $PT$ and $TP$ in
(\ref{eq:pt})--(\ref{eq:c2}). For this purpose, we introduce
amplitudes that may depend on four sets of variables, and
we will be able to choose different sets when dealing with
different composition formulae in the sequel.

\begin{lemma}\label{l:composition}
Let $\psi=\psi(x,\xi)\in C^\infty(\Rn\times\Rn)$ be such that
\begin{equation}\label{eq:psias1}
C_1\bra{\xi}\leq \bra{\nabla_x\psi(x,\xi)}\leq C_2\bra{\xi},
\textrm{ for all } x,\xi\in\Rn,
\end{equation}
for some $C_1,C_2>0$, and such that for all $|\alpha|, |\beta|\geq
1$ we have
\begin{equation}\label{eq:psias2}
|\partial_x^\alpha\psi(x,\xi)|\leq C_\alpha\bra{\xi}, \;\;
|\partial_x^\alpha\partial_\xi^\beta\psi(x,\xi)|\leq
C_{\alpha\beta},\textrm{ for all } x,\xi\in\Rn.
\end{equation}
Let $b=b(y,\xi,z,\eta)\in C^\infty(\Rone^{4n})$ satisfy
\begin{equation*}\label{eq:blong}
|\partial_y^\alpha\partial_\xi^\beta\partial_z^\gamma\partial_\eta^\delta
b(y,\xi,z,\eta)|\leq C_{\alpha\beta\gamma\delta}\bra{y}^{m_1}
\bra{\xi}^{m_2} \bra{z}^{m_3} \bra{\eta}^{m_4-|\delta|},
\end{equation*}
for all $\alpha, \beta, \gamma, \delta$, and all
$y,\xi,z,\eta\in\Rn$. Let $q=q(x,\eta)\in C^\infty(\Rn\times\Rn)$
for all $\alpha, \beta$ satisfy
\[
|\partial_x^\alpha\partial_\eta^\beta q(x,\eta)|\leq
C_{\alpha\beta}\bra{x}^{t_1} \bra{\eta}^{t_2-|\beta|},
\textrm{ for all } x,\eta\in\Rn.
\]
Then the function
\begin{equation*}
c(x,z,\xi)=\int_\Rn\int_\Rn
e^{i(\psi(y,\xi)-\psi(x,\xi)+(x-y)\cdot\eta)} b(y,\xi,z,\eta)
q(x,\eta) dy \dslash\eta
\end{equation*}
satisfies
\begin{equation}\label{eq:cest}
|\partial_x^\alpha\partial_z^\beta\partial_\xi^\gamma
c(x,z,\xi)|\leq C_{\alpha\beta\gamma}\bra{x}^{m_1+t_1}
 \bra{z}^{m_3}\bra{\xi}^{m_2+t_2+m_4},
\end{equation}
for all $\alpha, \beta, \gamma$, and all $x,z,\xi\in\Rn$.
Moreover, we have the asymptotic expansion, improving in $\xi$:
\begin{equation}\label{eq:asympt}
c(x,z,\xi) \thicksim \sum_{\alpha,\beta}
\frac{i^{-(|\alpha|+|\beta|)}}{\alpha!\beta!}
\partial_\xi^\alpha q(x,\nabla_x\psi(x,\xi))
\partial_y^{\alpha+\beta}\left[
e^{i\Psi(x,y,\xi)}\partial_\eta^\beta
b(y,\xi,z,\nabla_x\psi(x,\xi)) \right]|_{y=x},
\end{equation}
where
$\Psi(x,y,\xi)=\psi(y,\xi)-\psi(x,\xi)+(x-y)\cdot\nabla_x\psi(x,\xi).$
\end{lemma}

\begin{proof}
Let $\chi=\chi(x,y)\in C^\infty(\Rn\times\Rn)$ be such that
$0\leq\chi\leq 1$,
$\chi(x,y)=1$ for $|x-y|<k/2$ and $\chi(x,y)=0$ for $|x-y|>k$.
Let
\[
\begin{aligned}
c(x,z,\xi) & = c^I(x,z,\xi)+c^{II}(x,z,\xi) \\
& = \iint e^{i(\psi(y,\xi)-\psi(x,\xi)+(x-y)\cdot\eta)}
(1-\chi(x,y)) b(y,\xi,z,\eta) q(x,\eta) dy \dslash\eta \\
& + \iint e^{i(\psi(y,\xi)-\psi(x,\xi)+(x-y)\cdot\eta)} \chi(x,y)
b(y,\xi,z,\eta) q(x,\eta) dy \dslash\eta.
\end{aligned}
\]

{1. Estimate for $|x-y|\geq k$.} Here we will integrate by parts with
two operators,
\[
^tL_\eta=-i\sum_{j=1}^n \frac{x_j-y_j}{|x-y|^2}
\partial_{\eta_j}, \;\;
^tL_y=\frac{1-\Delta_y}{\bra{\nabla_y\psi(y,\xi)}^2-
i\Delta_y\psi(y,\xi)},
\]
so that we have
$L_\eta^{N_1}e^{i(x-y)\cdot\eta}=e^{i(x-y)\cdot\eta},$
$L_y^{N_2}e^{i\psi(y,\xi)}=e^{i\psi(y,\xi)}.$ Note that in view of
our assumption (\ref{eq:psias1}) on $\psi$, we have
\[
|\bra{\nabla_y\psi(y,\xi)}^2-i\Delta_y\psi(y,\xi)|\geq
|\bra{\nabla_y\psi(y,\xi)}|^2\geq C_1\bra{\xi}^2.
\]
Therefore,
\[
c^{I}(x,z,\xi) = \iint e^{i(\psi(y,\xi)-\psi(x,\xi)+x\cdot\eta))}
L_y^{N_2}\left( e^{-iy\cdot\eta} L_\eta^{N_1}\left[
(1-\chi(x,y))b(y,\xi,z,\eta) q(x,\eta) \right] \right) dy 
\dslash\eta.
\]
Further, we integrate by parts
with operator $(1+|x|^2)^{-1}(1-\Delta_\eta)$, for which
$(1+|x|^2)^{-1}(1-\Delta_\eta)e^{ix\cdot\eta}= e^{ix\cdot\eta}.$
This gives
\[
\begin{aligned}
c^{I}(x,z,\xi)=\iint e^{ix\cdot\eta}
e^{i(\psi(y,\xi)-\psi(x,\xi))} \bra{x}^{-2N_3}
(1-\Delta_\eta)^{N_3}
L_y^{N_2} \\ 
\left( e^{-iy\cdot\eta}
L_\eta^{N_1}\left[
(1-\chi(x,y))b(y,\xi,z,\eta) q(x,\eta)
\right]\right) dy \dslash\eta.
\end{aligned}
\]
Since $N_1$ $\eta$-derivatives of $b(y,\xi,z,\eta)q(x,\eta)$ in
view of our assumptions give an improvement in $\eta$ by
$\bra{\eta}^{-N_1}$, we get
\begin{equation}\label{EQ:est-cI}
\begin{aligned}
|c^{I}(x,z,\xi)| & \leq  \int \int_{|x-y|>k} \bra{x}^{-2N_3}
\bra{y}^{2N_3} \bra{\xi}^{-2N_2}
\bra{\eta}^{2N_2} \\
& \;\;\;\;\; |x-y|^{-2N_1}
\bra{y}^{m_1}\bra{\xi}^{m_2}\bra{z}^{m_3}\bra{x}^{t_1}
\bra{\eta}^{m_4+t_2-N_1} dy d\eta \\
& \leq  C \bra{z}^{m_3} \bra{x}^{t_1-2N_3} \bra{\xi}^{m_2-2N_2}
\int \bra{\eta}^{2N_2+m_4+t_2-N_1} d\eta
\int_{|x-y|>k} \frac{\bra{y}^{m_1+2N_3}}{|x-y|^{2N_1}} dy \\
& \leq  C \bra{z}^{m_3} \bra{x}^{t_1-2N_3} \bra{\xi}^{m_2-2N_2}
\int_{|y|>k} \frac{\bra{y+x}^{m_1+2N_3}}{|y|^{2N_1}} dy \\
& \leq  C \bra{z}^{m_3} \bra{x}^{t_1+m_1} \bra{\xi}^{m_2-2N_2},
\end{aligned}
\end{equation}
where we used $\bra{y+x}\leq C\bra{y}\bra{x}$ with $N_3$ such that
$m_1+2N_3>0$, and $2N_1> m_1+2N_3+n, 2N_2+m_4+t_2+n$, to make both
integrals finite.

Derivatives with respect to $z$ do not change estimates.
Derivatives with respect to $x$ and $\xi$ may introduce factors
estimated by powers of $\bra{\xi}, \bra{\eta}$ 
and $|y-x|$, which can be taken care of
by choosing large $N_1, N_2.$

{2. Estimate for $|x-y|<k.$} We will make a Taylor expansion of
$q(x,\eta)$ and $b(y,\xi,z,\eta)$ in
$\eta=\nabla_x\psi(x,\xi)+\theta$:
\begin{equation}\label{EQ:ralpha}
\begin{aligned}
q(x,\eta) & = \sum_{|\alpha|<M_1} \frac{\theta^\alpha}{\alpha !}
\partial_\xi^\alpha q(x,\nabla_x\psi(x,\xi)) + \sum_{|\alpha|=M_1}
C_\alpha {\theta^\alpha} r_\alpha(x,\xi,\theta), \\
b(y,\xi,z,\eta) & = \sum_{|\beta|<M_2} \frac{\theta^\beta}{\beta!}
\partial_\eta^\beta b(y,\xi,z,\nabla_x\psi(x,\xi)) + \sum_{|\beta|=M_2}
C_\beta {\theta^\beta} s_\beta(y,\xi,z,\theta), \\
r_\alpha(x,\xi,\theta) & = \int_0^1 (1-t)^{M_1}
\partial_\xi^{\alpha} q(x,\nabla_x\psi(x,\xi)+t\theta) dt, \\
s_\beta(y,\xi,z,\theta) & = \int_0^1 (1-t)^{M_2}
\partial_\xi^{\beta}b(y,\xi,z,\nabla_x\psi(x,\xi)+t\theta) dt.
\end{aligned}
\end{equation}
Let
\[
\Psi(x,y,\xi)=\psi(y,\xi)-\psi(x,\xi)+(x-y)\cdot\nabla_x\psi(x,\xi).
\]
Then
\[
\begin{aligned}
c^{II}(x,z,\xi) = \sum_{|\alpha|<M_1, |\beta|<M_2}
\frac{1}{\alpha !\beta !} c_{\alpha\beta}(x,z,\xi) +
 \sum_{|\alpha|=M_1, |\beta|=M_2}
C_\alpha C_\beta R_{\alpha\beta}(x,z,\xi),
\end{aligned}
\]
where
\[
\begin{aligned}
R_{\alpha\beta}(x,z,\xi) & = \iint e^{i(x-y)\cdot\theta}
e^{i\Psi(x,y,\xi)} \theta^{\alpha+\beta} \chi(x,y)
s_\beta(y,\xi,z,\theta) r_\alpha(x,\xi,\theta) dy 
\dslash\theta, \\
c_{\alpha\beta}(x,z,\xi) & = \iint e^{i(x-y)\cdot\theta}
\theta^{\alpha+\beta}\times \\
& \quad\quad \times\left[ e^{i\Psi(x,y,\xi)}\chi(x,y)
\partial_\eta^\beta b(y,\xi,z,\nabla_x\psi(x,\xi))
\partial_\xi^\alpha
q(x,\nabla_x\psi(x,\xi)) \right] dy \dslash\theta \\
& =
\partial_\xi^\alpha q(x,\nabla_x\psi(x,\xi))
D_y^{\alpha+\beta}\left[ e^{i\Psi(x,y,\xi)}\chi(x,y)
\partial_\eta^\beta b(y,\xi,z,\nabla_x\psi(x,\xi))
\right]|_{y=x}.
\end{aligned}
\]
Since $y=x$ is the zero of order two of $\Psi(x,y,\xi)$, it
follows that $\nabla_y e^{i\Psi(x,y,\xi)}|_{y=x}=0.$ Inductively,
it follows from (\ref{eq:psias2}) that $|\partial_y^{\gamma}
e^{i\Psi(x,y,\xi)}|_{y=x}|\leq C \bra{\xi}^{|\gamma|/2}.$ 
Hence we obtain
\[
\begin{aligned}
|c_{\alpha\beta}(x,z,\xi)| & \leq
C \bra{x}^{t_1} \bra{\nabla_x\psi(x,\xi)}^{t_2-|\alpha|}
\bra{\xi}^{\frac{|\alpha|+|\beta|}{2}}
\bra{x}^{m_1}\bra{\xi}^{m_2}\bra{z}^{m_3}
\bra{\nabla_x\psi(x,\xi)}^{m_4-|\beta|} \\
& \leq C\bra{x}^{m_1+t_1}\bra{z}^{m_3}
\bra{\xi}^{m_2+m_4+t_2-\frac{|\alpha|+|\beta|}{2}},
\end{aligned}
\]
where we have used that $\bra{\nabla_x\psi(x,\xi)}\sim \bra{\xi}$
by (\ref{eq:psias1}). Derivatives of $c_{\alpha\beta}$ with
respect to $x$ or $\xi$ do not change estimates when applied to
$q$ or $b$ by assumptions of the lemma. When applied to
$\partial_y^{\alpha+\beta} e^{i\Psi(x,y,\xi)}|_{y=x}$ they do not
change estimates since
$|\partial_\xi^\beta\partial_x^\alpha\psi(x,\xi)|\leq
C_{\alpha\beta}.$

{3. Estimates for the remainder $R_{\alpha\beta}$.} Let $\rho\in
C_0^\infty(\Rn)$ be such that $\rho(x)=1$ for $|x|<\epsilon/2$ and
$\rho(x)=0$ for $|x|>\epsilon$,
for some small $\epsilon>0$ to be chosen later. We decompose
\begin{equation}\label{EQ:Ralphas}
\begin{aligned}
R_{\alpha\beta}(x,z,\xi) & = R_\alpha^I(x,z,\xi)+
 R_\alpha^{II}(x,z,\xi) \\
 & = \iint e^{i(x-y)\cdot\theta}
\rho\left(\frac{\theta}{\bra{\xi}}\right)
D_y^{\alpha+\beta} \left[
\chi(x,y) s_\beta(y,\xi,z,\theta) r_\alpha(x,\xi,\theta) \right] dy 
\dslash\theta \\
& + \iint e^{i(x-y)\cdot\theta} \left(1-\rho\left(
\frac{\theta}{\bra{\xi}}\right)\right)
D_y^{\alpha+\beta}\left[ \chi(x,y) s_\beta(y,\xi,z,\theta)
r_\alpha(x,\xi,\theta)\right] dy \dslash\theta.
\end{aligned}
\end{equation}

{3.1. Estimate for $|\theta|\leq \epsilon\bra{\xi}.$} For
sufficiently small $\epsilon>0$, for any $0\leq t\leq 1$,
$\bra{\nabla_x\psi(x,\xi)+t\theta}$ and $\bra{\xi}$ are
equivalent. Indeed, if we use the inequality
\[
\bra{x}\leq 1+|x|\leq \sqrt{2}\bra{x},
\]
we get
\[
\begin{aligned}
\bra{\nabla_x\psi(x,\xi)+t\theta}\leq & (C_2\sqrt{2}+\epsilon)\bra{\xi} \\
\sqrt{2}\bra{\nabla_x\psi(x,\xi)+t\theta}\geq & 1+|\nabla_x\psi|-|\theta|
\geq \bra{\nabla_x\psi}-|\theta|\geq (C_1-\epsilon)\bra{\xi},
\end{aligned}
\]
so we will take $\epsilon<C_1.$ This equivalence means that for
$|\theta|\leq\epsilon\bra{\xi}$, functions
$r_\alpha(x,\xi,\theta)$, $s_\beta(y,\xi,z,\theta)$ and all of
their derivatives are dominated by $\bra{x}^{t_1}
\bra{\xi}^{t_2-|\alpha|}$ and
$\bra{y}^{m_1}\bra{\xi}^{m_2+m_4-|\beta|}\bra{z}^{m_3}$,
respectively. We will need two auxiliary estimates. The first
estimate

\begin{equation}\label{eq:estr}
\begin{aligned}
\left|\partial_\theta^{\gamma_1}\left( \rho\left(
\frac{\theta}{\bra{\xi}}\right)
r_\alpha(x,\xi,\theta)\right)\right| & \leq
C\sum_{\delta\leq\gamma_1} \left|\partial_\theta^\delta
\rho\left(\frac{\theta}{\bra{\xi}}\right)
\partial_\theta^{\gamma_1-\delta}
r_\alpha(x,\xi,\theta)\right| \\
& \leq C\bra{x}^{t_1}\sum_{\delta\leq\gamma_1}\bra{\xi}^{-|\delta|}
\bra{\xi}^{t_2-|\alpha|-|\gamma_1-\delta|} \\
& \leq C\bra{x}^{t_1}\bra{\xi}^{t_2-|\alpha|-|\gamma_1|}
\end{aligned}
\end{equation}
follows from the properties of $r_\alpha$. Before we state the
second estimate, let us analyse the structure of
$\partial_y^\alpha e^{i\Psi(x,y,\xi)}.$ It has at most $|\alpha|$
powers of terms $\nabla_y\psi(y,\xi)-\nabla_x\psi(x,\xi)$,
possibly also multiplied by at most $|\alpha|$ higher order
derivatives $\partial_y^\delta\psi(y,\xi)$, which can be estimated
by $(|y-x|\bra{\xi})^{|\alpha|}$ by (\ref{eq:psias2}). The term
containing no difference $\nabla_y\psi(y,\xi)-\nabla_x\psi(x,\xi)$
is the product of at most $|\alpha|/2$ terms of the type
$\partial_y^\delta\psi(y,\xi)$, which can be estimated by
$\bra{\xi}^{|\alpha|/2}$ in view of (\ref{eq:psias2}). Altogether,
we obtain the estimate
\[
|\partial_y^\alpha e^{i\Psi(x,y,\xi)}|\leq C_\alpha
(1+|y-x|\bra{\xi})^{|\alpha|} \bra{\xi}^{|\alpha|/2}.
\]
The second auxiliary estimate now is
\begin{equation}\label{eq:ests}
\begin{aligned}
& \left|D_y^{\alpha+\beta}\left[ e^{i\Psi(x,y,\xi)} \chi(x,y)
\partial_\theta^{\gamma_2}s_\beta(y,\xi,z,\theta)\right] \right|
\\ \leq &
(1+|y-x|\bra{\xi})^{|\alpha|+|\beta|}\bra{\xi}^{\frac{|\alpha|+|\beta|}{2}}
\bra{y}^{m_1}\bra{\xi}^{m_2}\bra{z}^{m_3}\bra{\xi}^{m_4-|\beta|-|\gamma_2|} \\
 \leq & (1+|y-x|\bra{\xi})^{|\alpha|+|\beta|}
\bra{y}^{m_1}\bra{z}^{m_3}
\bra{\xi}^{m_2+m_4+\frac{|\alpha|-|\beta|}{2}-|\gamma_2|}.
\end{aligned}
\end{equation}
Now we are ready to prove a necessary estimate for
$R^I_{\alpha\beta}(x,z,\xi).$ Let
\[
L_\theta=\frac{(1-\bra{\xi}^2\Delta_\theta)}{1+\bra{\xi}^2|x-y|^2},\;\;
L_\theta^N e^{i(x-y)\cdot\theta}=e^{i(x-y)\cdot\theta}.
\]
Integrations by parts with $L_\theta$ yield
\[
\begin{aligned}
R^I_{\alpha\beta}(x,z,\xi) = & \iint
\frac{e^{i(x-y)\cdot\theta}}{(1+\bra{\xi}^2 |x-y|^2)^N}
(1-\bra{\xi}^{2}\Delta_\theta)^N \\
& \;\; \left\{ \rho\left(\frac{\theta}{\bra{\xi}}\right)
D_y^{\alpha+\beta}\left[ \chi(x,y) s_\beta(y,\xi,z,\theta)
r_\alpha(x,\xi,\theta)\right]
\right\} dy \dslash\theta \\
= & \iint \frac{e^{i(x-y)\cdot\theta}}{(1+\bra{\xi}^2 |x-y|^2)^N}
\sum_{\stackrel{{\textrm some}\; \gamma_1,\gamma_2:}
{|\gamma_1|+|\gamma_2|\leq 2N}}
 c_{\gamma_1\gamma_2} \bra{\xi}^{|\gamma_1|+|\gamma_2|} \\
& \;\;
\left\{
D_y^{\alpha+\beta}\left[
\chi(x,y) \partial_\theta^{\gamma_2}s_\beta(y,\xi,z,\theta)
\partial_\theta^{\gamma_1}\left(
\rho\left(\frac{\theta}{\bra{\xi}}\right)
r_\alpha(x,\xi,\theta)\right)\right]
\right\} dy \dslash\theta.
\end{aligned}
\]
Using estimates (\ref{eq:estr}), (\ref{eq:ests})
for $r_\alpha$, $s_\beta$, and the fact that the
size of the support of $\rho(\theta/\bra{\xi})$ in $\theta$ is
estimated by $(\epsilon\bra{\xi})^n$, we obtain the estimate
\[
\begin{aligned}
|R^I_{\alpha\beta}(x,z,\xi)| & \leq
C\sum_{|\gamma_1|+|\gamma_2|\leq 2N}
\bra{\xi}^{n+|\gamma_1|+|\gamma_2|}\bra{x}^{t_1}
\bra{\xi}^{t_2-|\alpha|-|\gamma_1|} \bra{z}^{m_3}
\bra{\xi}^{m_2+m_4+\frac{|\alpha|-|\beta|}{2}-|\gamma_2|}\times \\
& \;\;\;\; \times
\int_{|y-x|<k} \bra{y}^{m_1}\frac{(1+|y-x|\bra{\xi})^{|\alpha|+|\beta|}}
{(1+\bra{\xi}^2 |x-y|^2)^N} dy \\
& \leq C\bra{x}^{t_1+m_1}\bra{\xi}^{t_2+m_2+m_4+2n-\frac{|\alpha|+|\beta|}{2}}
\bra{z}^{m_3},
\end{aligned}
\]
if we choose $2N>M_1+M_2+n.$ In the last estimate we have used that
for $|x-y|<k$, $\bra{x}$ and $\bra{y}$ are equivalent for large $x$.

Derivatives of $R_{\alpha\beta}^I(x,z,\xi)$ with respect to $x$
and $\xi$ give an extra power of $\theta$ under the integral. This
amounts to taking more $y$-derivatives, giving a higher power of
$\bra{\xi}.$ However, this is not a problem if for the estimate
for a given number of derivatives of the remainder
$R_{\alpha\beta}^I(x,z,\xi)$, we choose $M_1=|\alpha|$ and 
$M_2=|\beta|$
sufficiently large.

{3.2. Estimate for $|\theta|>\epsilon\bra{\xi}.$}
Let us define
\[
\omega(x,y,\xi,\theta)=(x-y)\cdot\theta+\Psi(x,y,\xi)=
(x-y)\cdot(\nabla_x\psi(x,\xi)+\theta)+\psi(y,\xi)-\psi(x,\xi).
\]
From (\ref{eq:psias1}) and (\ref{eq:psias2}) we obviously have
\begin{equation}\label{eq:rho}
\begin{aligned}
|\nabla_y\omega| & =|-\theta+\nabla_y\psi-\nabla_x\psi|\leq
2C_2(|\theta|+\bra{\xi}), \\
|\nabla_y\omega| & \geq |\theta|-|\nabla_y\psi-\nabla_x\psi| \geq
\frac{1}{2}|\theta|+\left(\frac{\epsilon}{2}-C_0|x-y|
\right)\bra{\xi}\geq
C(|\theta|+\bra{\xi}),
\end{aligned}
\end{equation}
if we choose $k<\epsilon/2C_0,$ since $|x-y|<k$ in the support of $\chi$.
Let us denote
\[
\sigma_{\gamma_1}(x,y,\xi):=e^{-i\Psi(x,y,\xi)}D_y^{\gamma_1}
e^{i\Psi(x,y,\xi)},\;
\chi_{\gamma_2}(x,y):=D_y^{\gamma_2}\chi(x,y),\;
s_{\beta\gamma_3}:=D_y^{\gamma_3}s_\beta.
\]
For any $\nu$ we have an estimate
$|\partial_y^\nu\sigma_{\gamma_1}(x,y,\xi)|\leq\bra{\xi}^{|\gamma_1|}$,
because of our assumption (\ref{eq:psias2}) that
$|\partial_y^\nu\psi(y,\xi)|\leq C_\nu\bra{\xi}.$ For
$M_1=|\alpha|>t_2$ and $M_2=|\beta|>m_4$, we also observe that
\begin{equation}\label{eq:rs}
|r_\alpha(x,\xi,\theta)|\leq C_\alpha\bra{x}^{t_1}, \;
|\partial_y^\nu s_{\beta\gamma_3}(y,\xi,z,\theta)|\leq
C_\beta\bra{y}^{m_1}\bra{\xi}^{m_2} \bra{z}^{m_3}.
\end{equation}
Let us take $^tL_y=i|\nabla_y\omega|^{-2}\sum_{j=1}^n
(\partial_{y_j}\omega)
\partial_{y_j}.$ It can be shown by induction that operator
$L_y^N$ has the form
\begin{equation}\label{EQ:LNy}
L_y^N=\frac{1}{|\nabla_y\omega|^{4N}}\sum_{|\nu|\leq N}
P_{\nu,N}\partial_y^\nu,\;\; P_{\nu,N}=\sum_{|\mu|=2N}
c_{\nu\mu\delta_j}(\nabla_y\omega)^\mu
\partial_y^{\delta_1}\omega\cdots
\partial_y^{\delta_N}\omega,
\end{equation}
where $|\mu|=2N, |\delta_j|\geq 1, \sum_1^N |\delta_j|+|\nu|=2N.$
It follows from (\ref{eq:psias2}) that $|P_{\nu,N}|\leq
C(|\theta|+\bra{\xi})^{3N}.$ By Leibnitz formula we have
\begin{equation}\label{EQ:RII}
\begin{aligned}
& R^{II}_{\alpha\beta}(x,z,\xi) = \\
& = \iint e^{i(x-y)\cdot\theta} \left(1-\rho\left(
\frac{\theta}{\bra{\xi}}\right)\right)
r_\alpha(x,\xi,\theta) \times \\
& \;\;\; \times D_y^{\alpha+\beta}\left[
e^{i\Psi(x,y,\xi)}\chi(x,y)
s_\beta(y,\xi,z,\theta)\right] dy \dslash\theta \\
& = \iint e^{i\omega(x,y,\xi,\theta)}
\left(1-\rho\left(\frac{\theta}{\bra{\xi}}\right)\right) 
r_\alpha(x,\xi,\theta) \times
\\
& \; \; \; \times \sum_{\gamma_1+\gamma_2+\gamma_3=\alpha+\beta}
\sigma_{\gamma_1}(x,y,\xi)\chi_{\gamma_2}(x,y)
s_{\beta\gamma_3}(y,\xi,z,\theta) dy \dslash\theta \\
& = \iint e^{i\omega(x,y,\xi,\theta)}
|\nabla_y\omega|^{-4N}\sum_{|\nu|\leq N} P_{\nu,N}(x,y,\xi,\theta)
\left(1-\rho\left(\frac{\theta}{\bra{\xi}}
\right)\right) r_\alpha(x,\xi,\theta)\times \\
& \;\;\; \times \sum_{\gamma_1+\gamma_2+\gamma_3=\alpha+\beta}
\partial_y^\nu\left(\sigma_{\gamma_1}(x,y,\xi)\chi_{\gamma_2}(x,y)
s_{\beta\gamma_3}(y,\xi,z,\theta)\right) dy \dslash\theta.
\end{aligned}
\end{equation}
It follows now from (\ref{eq:rs}) that
\[
\begin{aligned}
|R^{II}_{\alpha\beta}(x,z,\xi)| & \leq C
\int_{|\theta|>\epsilon\bra{\xi}} \int_{|y-x|<k}
(|\theta|+\bra{\xi})^{-N} \bra{x}^{t_1}
\bra{\xi}^{|\alpha|+|\beta|}
\bra{y}^{m_1}\bra{\xi}^{m_2}\bra{z}^{m_3} dy d\theta \\
& \leq C \bra{x}^{t_1+m_1}\bra{z}^{m_3}\bra{\xi}^{m_2+|\alpha|+|\beta|}
\int_{|\theta|>\epsilon\bra{\xi}}|\theta|^{-N}d\theta \\
& \leq C \bra{x}^{t_1+m_1}\bra{z}^{m_3}
\bra{\xi}^{m_2+|\alpha|+|\beta|+n-N},
\end{aligned}
\]
which yields the desired estimate for $N>m_2+|\alpha|+|\beta|+n.$
For the derivatives of $R_{\alpha\beta}^{II}(x,z,\xi)$, similar to
Part 3.1 for $R_{\alpha\beta}^I$, we can get an extra power of
$\theta$, which can be taken care of by choosing large $N$ and
using the fact that $|x-y|<k.$ The proof of Lemma
\ref{l:composition} is now complete.
\end{proof}

In the proof of the second part of Theorem \ref{th:sob1}
we will need to have stronger decay properties for the 
amplitudes, so now we will formulate the corresponding
counterpart of Lemma \ref{l:composition}:
\begin{proposition}\label{prop-add}
Under conditions of Lemma \ref{l:composition},
assume that stronger estimates
\begin{equation}\label{EQ:add-dec}
|\partial_y^\alpha\partial_\xi^\beta\partial_z^\gamma
\partial_\eta^\delta
b(y,\xi,z,\eta)|\leq C_{\alpha\beta\gamma\delta}\bra{y}^
{m_1-|\alpha|}
\bra{\xi}^{m_2} \bra{z}^{m_3} \bra{\eta}^{m_4-|\delta|}
\end{equation}
and 
\begin{equation}\label{EQ:add-decp}
|\partial_x^\alpha\partial_\eta^\beta q(x,\eta)|\leq
C_{\alpha\beta}\bra{x}^{t_1-|\alpha|} \bra{\eta}^{t_2-|\beta|}
\end{equation}
hold.
Assume also that phase function $\psi(x,\xi)$ satisfies in
addition estimates
\begin{equation}\label{EQ:add-dec-phase}
|\partial_x^\alpha\psi(x,\xi)|\leq C_{\alpha}
\jp{x}^{1-|\alpha|}\jp{\xi},\quad
|\partial_x^\alpha\partial_\xi^\beta\psi(x,\xi)|\leq 
C_{\alpha\beta}\jp{x}^{1-|\alpha|},
\end{equation}
for all multi-indices $\alpha$ and $\beta$ with $|\beta|\geq 1$,
and all $x,\xi\in\Rn$.
Then
we can improve estimate \eqref{eq:cest} as
\begin{equation}\label{EQ:add-dec2}
|\partial_x^\alpha\partial_z^\beta\partial_\xi^\gamma
c(x,z,\xi)|\leq C_{\alpha\beta\gamma}\bra{x}^{m_1+t_1-|\alpha|}
\bra{z}^{m_3}\bra{\xi}^{m_2+t_2+m_4}.
\end{equation}
\end{proposition}

\begin{proof}[Proof of Proposition \ref{prop-add}]
In the case of proving estimate \eqref{EQ:add-dec2}, we have to
go back and modify the proof of Lemma \ref{l:composition} to
adopt it to the required decay properties. We will
assume the notation of the proof of Lemma \ref{l:composition} 
and will indicate the differences for the present case.
The main difference 
is that now we will
work with a different cut-off function $\chi=\chi(x,y)\in 
C^\infty(\Rn\times\Rn)$
such that
$0\leq\chi\leq 1$,
$\chi(x,y)=1$ for $|x-y|<\frac{k}{2}\jp{x}$ and 
$\chi(x,y)=0$ for $|x-y|>k\jp{x}$, for some small
$0<k<1$ to be chosen later.

1. Estimate for $|x-y|\geq k\jp{x}$. The proof goes similar to
the corresponding proof of the first part of 
 Lemma \ref{l:composition}. In the estimate for $c^I$
 in \eqref{EQ:est-cI} we observe that the domain of the
 integration in the last integral changes from
 $|y|>k$ to $|y|>k\jp{x}$, thus introducing the additional
 factor of $\jp{x}^{-N}$ for any $N$, if we choose
 $2N_1>m_1+2N_3+n+N$. Therefore, the part $c^I$ of the amplitude
 $c$ is actually decaying as any negative power 
 of $\jp{x}$ in this case.
 
2. Estimates for $|x-y|<k\jp{x}$. In Part 2 of the proof of 
Lemma \ref{l:composition} we established the formula
for the terms in the asymptotic expansion of $c^{II}$, given by
\[
c_{\alpha\beta}(x,z,\xi) =
\partial_\xi^\alpha q(x,\nabla_x\psi(x,\xi))
D_y^{\alpha+\beta}\left[ e^{i\Psi(x,y,\xi)}\chi(x,y)
\partial_\eta^\beta b(y,\xi,z,\nabla_x\psi(x,\xi))
\right]|_{y=x}.
\]
Now, differentiation of 
$\partial_\xi^\alpha q(x,\nabla_x\psi(x,\xi))$ with respect to 
$x$ brings the decay by $\jp{x}^{-1}$ in view of assumptions
\eqref{EQ:add-decp} and \eqref{EQ:add-dec-phase}. Differentiation
of $\chi$ and $b$ with respect to both $x$ and
$y$ also gives a factor
of $\jp{x}^{-1}$ on the support of $\chi$. Finally, since
$$\Psi(x,y,\xi)=\phi(y,\xi)-\phi(x,\xi)+
(x-y)\cdot\nabla_x\phi(x,\xi),$$ 
we have
$\nabla_y e^{i\Psi(x,y,\xi)}|_{y=x}=0$ and 
$D_y^{\alpha}\Psi(x,y,\xi)|_{y=x}=D_x^\alpha\phi(x,\xi)$
for $|\alpha|\geq 2$. Thus, by the product rule applied to
the derivatives $D_y^{\alpha+\beta}$,
function $c_{\alpha\beta}$ contains derivatives of
$\chi$ and $b$, multiplied by terms of the form
$D_x^\alpha\phi(x,\xi)$ coming from the differentiation of
the exponential $e^{i\Psi(x,y,\xi)}$ at $y=x$.
Differentiation of such terms with respect to $x$
also brings the decay of order $\jp{x}^{-1}$, thus
establishing the required decay property \eqref{EQ:add-dec2}
for functions $c_{\alpha\beta}$.

3. Estimates for the remainder $R_{\alpha\beta}$.
 Here again we follow the proof of Part 3 of 
 Lemma \ref{l:composition}, adopting all its notation.

 3.1. In the region $|\theta|\leq\epsilon\jp{\xi}$, the definition
 of $r_\alpha$ in \eqref{EQ:ralpha} and assumptions
 \eqref{EQ:add-decp} and \eqref{EQ:add-dec-phase}
 imply that $\partial_x^\gamma r_\alpha$ brings an 
 additional decay by $\jp{x}^{-|\gamma|}$ compared to
 that of $r_\alpha$. Now, differentiation of $R_{\alpha\beta}^I$ in
 \eqref{EQ:Ralphas} gives this decay when the $x$--derivatives
 fall on $r_\alpha$. If the $x$--derivatives fall on the
 exponential $e^{i(x-y)\cdot\theta}$, they give additional
 powers of $\theta$, which transform into $y$--derivatives
 once we integrate by parts with respect to $y$. 
 Since $y$--derivatives of
 both $\chi$ and $s_\beta$ bring decay in $y$, 
 and hence also in $\jp{x}$ in view of $\jp{x}\leq C\jp{y}$
 on the support of $\chi$ (we choose $k$ to be sufficiently
 small),
 we establish
 \eqref{EQ:add-dec2} for functions $R_{\alpha\beta}^I$.
 
  3.2. In the region $|\theta|>\epsilon\jp{\xi}$,
  we note first that the estimate $|\nabla_y\omega|\geq 
 C(|\theta|+\jp{\xi})$ in \eqref{eq:rho} is still
 valid in view of \eqref{EQ:add-dec-phase}, if we choose
 $k>0$ to be sufficiently small compared to $\epsilon$.
  
  Now, we will show that under assumptions of Proposition 
  \ref{prop-add}, $R_{\alpha\beta}^{II}$
  is actually of order $\jp{x}^{-N}$ in $x$, for any $N$.
  We note that we have
  $|\partial_y^\delta\omega|\leq C_{\delta}\jp{y}^{1-|\delta|}
  (|\theta|+\jp{\xi}),$ so that function $P_{\nu,N}$ in
  \eqref{EQ:LNy} can be estimated by
  $|P_{\nu,N}|\leq C\jp{y}^{|\nu|-N}(|\theta|+\jp{\xi})^{3N}$,
  since $\sum_1^N (1-|\delta_j|)=N-(2N-|\nu|)=|\nu|-N.$
  On the other hand, functions
  $\sigma_{\gamma_1}, \chi_{\gamma_2}$ and $s_{\beta\gamma_3}$
  give the improvement by $\jp{y}^{-|\nu|}$ when 
  differentiated $\nu$ times with respect to $y$.
  It follows that
  $\partial_y^\nu(\sigma_{\gamma_1}\chi_{\gamma_2} 
  s_{\beta\gamma_3})$ has growth in $y$ of at most
  $\jp{y}^{m_3-|\nu|}$. Consequently, we obtain that the
  function under the last integral in \eqref{EQ:RII}
  is of order $\jp{y}^{(|\nu|-N)+(m_3-|\nu|)}=\jp{y}^{m_3-N}$ in 
  $y$. Since $y$ is equivalent to $x$ on the support of $\chi$
  for small $k$, we see from \eqref{EQ:RII} 
  that $R_{\alpha\beta}^{II}$ satisfies the estimate
 \begin{equation}\label{EQ:Riie}
  |R^{II}_{\alpha\beta}(x,z,\xi)|  \leq C
  \bra{x}^{t_1+m_1-N}\bra{z}^{m_3}
  \bra{\xi}^{m_2+|\alpha|+|\beta|+n-N}.
 \end{equation}
 Now, let us show that differentiations (especially with
 respect to $x$) guarantee orders of other variables as
 in \eqref{EQ:add-dec2}.
 The definition
 of $r_\alpha$ in \eqref{EQ:ralpha} and assumptions
 \eqref{EQ:add-decp} and \eqref{EQ:add-dec-phase}
 imply that $\partial_x^\gamma r_\alpha$ introduces an 
 additional factor bounded by 
 $\jp{x}^{-|\gamma|}\jp{\xi}^{|\gamma|}$ compared to
 that of $r_\alpha$.
 Differentiating 
 $R^{II}_{\alpha\beta}$ in \eqref{EQ:RII} 
 with respect to $x$, by the product rule we get that
 if the derivatives fall on either $r_\alpha$ or $\chi$,
 we get the additional decay in $\jp{x}$ times possible
 growth in $\jp{\theta}$. If they fall on 
 the exponential 
 $e^{i(x-y)\cdot\theta}$, we integrate by parts with 
 respect to $y$ to get the decay in $\jp{y}$ times
 possible (polynomial)
 growth in $\jp{\theta}$. Since $\jp{y}\geq C\jp{x}$
 on the support of $\chi$, in both cases we get the
 improvement in $\jp{x}$ compensated by possible growth of
 the same order in $\jp{\theta}$. Derivatives of $e^{i\Psi(x,y,\xi)}$
 may give growth in $x,y$ or $\xi$, but this is compensated
 by the arbitrarily fast decay in these variables.
 Altogether, since in \eqref{EQ:Riie} and in \eqref{EQ:RII}
 we get the overall decay of $R^{II}_{\alpha\beta}$
 to be of any desired order in all $x$,
 $\xi$ and $\theta$, this compensates possible
 additional growth that we obtain by differentiating
 $R^{II}_{\alpha\beta}$. This
 completes the proof of \eqref{EQ:add-dec2}.
\end{proof}

Applying Lemma \ref{l:composition} to (\ref{eq:c1}) and
(\ref{eq:c2}), we obtain two composition theorems.

\begin{theorem}[Composition $PT$]\label{th:pt}
Let operator $T$ be defined by
\[
Tu(x)=\int_\Rn\int_\Rn e^{i(\phi(x,\xi)-y\cdot\xi)} a(x,y,\xi)
u(y) dy \dslash\xi.
\]
Let the phase $\phi=\phi(x,\xi)\in C^\infty(\Rn\times\Rn)$ be such
that
\begin{equation*}\label{eq:phase11}
C_1\bra{\xi}\leq \bra{\nabla_x\phi(x,\xi)}\leq C_2\bra{\xi},
\textrm{ for all } x,\xi\in\Rn,
\end{equation*}
for some $C_1,C_2>0$, and such that for all $|\alpha|, |\beta|\geq
1$ we have
\begin{equation*}\label{eq:phase12}
|\partial_x^\alpha\phi(x,\xi)|\leq C_\alpha\bra{\xi}, \;\;
|\partial_x^\alpha\partial_\xi^\beta\phi(x,\xi)|\leq
C_{\alpha\beta},\textrm{ for all } x,\xi\in\Rn.
\end{equation*}
Let $a=a(x,y,\xi)\in C^\infty(\Rone^{n}\times\Rn\times\Rn)$ satisfy
\begin{equation}\label{eq:a1}
|\partial_x^\alpha\partial_y^\beta\partial_\xi^\gamma
a(x,y,\xi)|\leq C_{\alpha\beta\gamma}\bra{x}^{m_1} \bra{y}^{m_2}
\bra{\xi}^{m_3},
\end{equation}
for all $\alpha, \beta, \gamma$, and all $x,y,\xi\in\Rn$. Let
$p=p(x,\xi)\in C^\infty(\Rn\times\Rn)$ for all $\alpha, \beta$
satisfy
\begin{equation}\label{eq:a1p}
|\partial_x^\alpha\partial_\xi^\beta p(x,\xi)|\leq
C_{\alpha\beta}\bra{x}^{t_1} \bra{\xi}^{t_2-|\beta|},
\textrm{ for all } x,\xi\in\Rn.
\end{equation}
Then the composition $B=P(x,D)\circ T$ is an operator of the form
\[
Bu(x)=\int_\Rn e^{i(\phi(x,\xi)-z\cdot\xi)} c(x,z,\xi) u(z) dz
\dslash\xi,
\]
with amplitude $c(x,z,\xi)$ satisfying
\begin{equation}\label{eq:newamp1}
|\partial_x^\alpha\partial_z^\beta\partial_\xi^\gamma
c(x,z,\xi)|\leq C_{\alpha\beta\gamma}\bra{x}^{m_1+t_1}
\bra{z}^{m_2}\bra{\xi}^{m_3+t_2},
\end{equation}
for all $\alpha, \beta, \gamma$, and all $x,z,\xi\in\Rn$.
Moreover, we have the asymptotic expansion, improving in $\xi$:
\begin{equation*}\label{eq:asympt1}
c(x,z,\xi) \thicksim \sum_{\alpha} \frac{i^{-|\alpha|}}{\alpha !}
\partial_\xi^\alpha p(x,\nabla_x\phi(x,\xi))
\partial_y^{\alpha}\left[
e^{i\Psi(x,y,\xi)}a(y,z,\xi) \right]|_{y=x},
\end{equation*}
where
$\Psi(x,y,\xi)=\phi(y,\xi)-\phi(x,\xi)+(x-y)\cdot\nabla_x\phi(x,\xi).$
\end{theorem}
\begin{proof}[Proof of Theorem \ref{th:pt}]
Formula for $c(x,z,\xi)$ is given in (\ref{eq:pt}) and
(\ref{eq:c1}). The rest follows from Lemma \ref{l:composition}
with $\psi=\phi$, $q=p$, and $b(y,\xi,z,\eta)=a(y,z,\xi),$ with a
different notation for the orders. Asymptotic expansion follows
from (\ref{eq:asympt}) if we notice that we only have terms with
$\beta=0$.
\end{proof}

\begin{rem}\label{rem2}
As in Remark \ref{rem1}, an additional
decay of $a$ implies the same decay property of new amplitude.
For example, if we assume stronger estimate
\begin{equation*}\label{eq:sta1}
|\partial_x^\alpha\partial_y^\beta\partial_\xi^\gamma
a(x,y,\xi)|\leq C_{\alpha\beta\gamma}\bra{x}^{m_1} \bra{y}^{m_2-|\beta|}
\bra{\xi}^{m_3}
\end{equation*}
instead of \eqref{eq:a1}, then
we can improve estimate \eqref{eq:newamp1} as
\begin{equation}\label{eq:stnewamp1}
|\partial_x^\alpha\partial_z^\beta\partial_\xi^\gamma
c(x,z,\xi)|\leq C_{\alpha\beta\gamma}\bra{x}^{m_1+t_1}
\bra{z}^{m_2-|\beta|}\bra{\xi}^{m_3+t_2}.
\end{equation}
In fact, from the expression \eqref{eq:c1}, we obtain
\[
\jp{z}^{\beta}\partial_z^\beta c(x,z,\xi)=
\int_\Rn\int_\Rn
 e^{i(\phi(y,\xi)-\phi(x,\xi)+(x-y)\cdot\eta)}
\jp{z}^{\beta}(\partial_z^\beta a)(y,z,\xi) p(x,\eta) dy \dslash\eta
\]
and the conclusion \eqref{eq:stnewamp1} follows from \eqref{eq:newamp1}
if we regard
$\jp{y}^{\beta}(\partial_y^\beta a)(x,y,\xi)$
as $a(x,y,\xi)$ in \eqref{eq:a1}.
\end{rem}
For the additional decay properties with respect to $x$ we
have the following:
\begin{rem}\label{rem3}
Under conditions of Theorem \ref{th:pt},
assume that stronger estimates
\begin{equation*}
|\partial_x^\alpha\partial_y^\beta\partial_\xi^\gamma
a(x,y,\xi)|\leq C_{\alpha\beta\gamma}\bra{x}^{m_1-|\alpha|} \bra{y}^{m_2}
\bra{\xi}^{m_3}
\end{equation*}
and 
\begin{equation*}\label{EQ:p-add-decp}
|\partial_x^\alpha\partial_\xi^\beta p(x,\xi)|\leq
C_{\alpha\beta}\bra{x}^{t_1-|\alpha|} \bra{\xi}^{t_2-|\beta|}
\end{equation*}
hold instead of \eqref{eq:a1} and \eqref{eq:a1p}.
Assume also that phase function $\phi(x,\xi)$ satisfies in
addition estimates
\begin{equation*}\label{EQ:p-add-dec-phase}
|\partial_x^\alpha\phi(x,\xi)|\leq C_{\alpha}
\jp{x}^{1-|\alpha|}\jp{\xi},\quad
|\partial_x^\alpha\partial_\xi^\beta\phi(x,\xi)|\leq 
C_{\alpha\beta}\jp{x}^{1-|\alpha|},
\end{equation*}
for all multi-indices $\alpha$ and $\beta$ with $|\beta|\geq 1$,
and all $x,\xi\in\Rn$.
Then
we can improve estimate \eqref{eq:newamp1} as
\begin{equation}\label{EQ:p-add-dec2}
|\partial_x^\alpha\partial_z^\beta\partial_\xi^\gamma
c(x,z,\xi)|\leq C_{\alpha\beta\gamma}\bra{x}^{m_1+t_1-|\alpha|}
\bra{z}^{m_2}\bra{\xi}^{m_3+t_2}.
\end{equation}
\end{rem}
To prove the statement of Remark \ref{rem3} we use
Proposition \ref{prop-add} with 
$\psi=\phi$, $q=p$, and $b(y,\xi,z,\eta)=a(y,z,\xi).$
Assumptions of Remark  \ref{rem3} then guarantee that
assumptions of Proposition \ref{prop-add} are satisfied,
so \eqref{EQ:add-dec2} implies \eqref{EQ:p-add-dec2}.

\bigskip
The composition formula for $TP$ is given in Theorem
\ref{prop:alt} under less restrictive conditions on the phase.
However, the following theorem shows that if the amplitude
$a(x,y,\xi)$ has decay properties in $y$, the amplitude of the
composition $TP$ can be made dependent on two variables only.

\begin{theorem}[Composition $TP$]\label{th:tp}
Let operator $T$ be defined by
\[
Tu(x)=\int_\Rn\int_\Rn e^{i(\phi(x,\xi)-y\cdot\xi)} a(x,y,\xi)
u(y) dy \dslash\xi.
\]
Let the phase $\phi=\phi(x,\xi)\in C^\infty(\Rn\times\Rn)$ be such
that
\begin{equation}\label{eq:phase21}
C_1\bra{x}\leq \bra{\nabla_\xi\phi(x,\xi)}\leq C_2\bra{x},
\textrm{ for all } x,\xi\in\Rn,
\end{equation}
for some $C_1,C_2>0$, and such that for all $|\alpha|, |\beta|\geq
1$ we have
\begin{equation}\label{eq:phase22}
|\partial_\xi^\beta\phi(x,\xi)|\leq C_\beta\bra{x}, \;\;
|\partial_x^\alpha\partial_\xi^\beta\phi(x,\xi)|\leq
C_{\alpha\beta},\textrm{ for all } x,\xi\in\Rn.
\end{equation}
Let $a=a(x,y,\xi)\in C^\infty(\Rone^{n}\times\Rn\times\Rn)$ satisfy
\begin{equation}\label{eq:a2}
|\partial_x^\alpha\partial_y^\beta\partial_\xi^\gamma
a(x,y,\xi)|\leq C_{\alpha\beta\gamma}\bra{x}^{m_1}
\bra{y}^{m_2-|\beta|} \bra{\xi}^{m_3},
\end{equation}
for all $\alpha, \beta, \gamma$, and all $x,y,\xi\in\Rn$. Let
$p=p(x,\xi)\in C^\infty(\Rn\times\Rn)$ for all $\alpha, \beta$
satisfy
\[
|\partial_x^\alpha\partial_\xi^\beta p(x,\xi)|\leq
C_{\alpha\beta}\bra{x}^{t_1-|\alpha|} \bra{\xi}^{t_2},
\textrm{ for all } x,\xi\in\Rn.
\]
Then the composition $B=T\circ P(x,D)$ is an operator of the form
\[
Bu(x)=\int_\Rn e^{i\phi(x,\xi)} c(x,\xi) \widehat{u}(\xi) 
\dslash\xi,
\]
with amplitude $c(x,\xi)$ satisfying
\begin{equation*}\label{eq:newamp2}
|\partial_x^\alpha\partial_\xi^\beta c(x,\xi)|\leq
C_{\alpha\beta}\bra{x}^{m_1+m_2+t_1} \bra{\xi}^{m_3+t_2},
\end{equation*}
for all $\alpha, \beta, \gamma$, and all $x,z,\xi\in\Rn$.
Moreover, we have the asymptotic expansion, improving in $x$:

\begin{equation*}\label{eq:asympt2}
c(x,\xi) \thicksim \sum_{\alpha,\beta}
\frac{i^{-(|\alpha|+|\beta|)}}{\alpha !\beta !}
\partial_x^\alpha p(\nabla_\xi\phi(x,\xi),\xi)
\partial_\eta^{\alpha+\beta}\left[
e^{i\Psi(\eta,\xi,x)}\partial_y^\beta
a(x,\nabla_\xi\phi(x,\xi),\eta) \right]|_{\eta=\xi},
\end{equation*}
where
$\Psi(\eta,\xi,x)=\phi(x,\xi)-\phi(x,\eta)+(\eta-\xi)\cdot\nabla_\xi\phi(x,\xi).$
\end{theorem}
\begin{proof}
Let us rename the variables
\[
x^\prime=\eta, y^\prime=\xi, \xi^\prime=x, \eta^\prime=y.
\]
Then formula (\ref{eq:c2}) becomes
\[
c_2(\xi^\prime,x^\prime)=\iint
e^{i(\psi(y^\prime,\xi^\prime)-\psi(x^\prime,\xi^\prime)+(x^\prime-y^\prime)
\cdot \eta^\prime)} b(y^\prime,\xi^\prime,\eta^\prime)
q(x^\prime,\eta^\prime) dy^\prime \dslash\eta^\prime,
\]
if we choose $\psi(x^\prime,\xi^\prime)=\phi(\xi^\prime,x^\prime),
q(x^\prime,\eta^\prime)=p(\eta^\prime,x^\prime),$ and
$b(y^\prime,\xi^\prime,z,\eta^\prime)=a(\xi^\prime,\eta^\prime,y^\prime).$
Then we can apply Lemma \ref{l:composition} to complete the proof.
\end{proof}
If we apply Theorem \ref{th:tp} with $p=1$, we can simplify the
amplitude of the operator $T$. 
Operators of such type is a simplified version of
Fourier integral operators and the local versions are described
by, for example, Kumano-Go \cite{Ku1}, Stein \cite{St}, and many
other authors. This allows one to make representations for
operators as oscillatory integrals to depend on fewer
variables.

\begin{cor}\label{cor:repr}
Let $T$ be the operator from Theorem \ref{th:tp} with phase $\phi$
satisfying (\ref{eq:phase21}), (\ref{eq:phase22}), and amplitude
satisfying (\ref{eq:a2}). Then $T$ can be written in the form
\[
Tu(x)=\int_\Rn e^{i\phi(x,\xi)} c(x,\xi) \widehat{u}(\xi) 
\dslash\xi,
\]
with amplitude $c(x,\xi)$ satisfying
\[
|\partial_x^\alpha\partial_\xi^\beta c(x,\xi)|\leq
C_{\alpha\beta}\bra{x}^{m_1+m_2} \bra{\xi}^{m_3},
\]
for all $\alpha, \beta$, and all $x,\xi\in\Rn$.
Moreover,
\begin{equation*}\label{eq:asymptcor}
c(x,\xi) \thicksim \sum_{\beta} \frac{i^{-|\beta|}}{\beta !}
\partial_\eta^{\beta}\left[
e^{i\Psi(\eta,\xi,x)}\partial_y^\beta
a(x,y,\eta)|_{y=\nabla_\xi\phi(x,\xi)} \right]|_{\eta=\xi},
\end{equation*}
where $\Psi$ is as in Theorem \ref{th:tp}.
\end{cor}

We can apply Corollary \ref{cor:repr} to pseudo-differential
operators to obtain a the standard quantisation
for pseudo-differential operators with amplitudes 
having decay in only one of the
variables. Different quantisations of operators are
possible, including symbols of the type
$c(tx+(1-t)y,\xi)$ for $0\leq t\leq 1$, with Weyl
quantization given by $t=1/2$. However, for the purposes
of this paper and for simplicity
we will restrict our attention to the
case $t=1$. Other cases may be done with small
modifications.

\begin{cor}\label{cor:repr1}
Let $T$ be an operator of the form
\[
Tu(x)=\int_\Rn\int_\Rn e^{i(x-y)\cdot\xi} a(x,y,\xi)
u(y)\,dy\,\dslash\xi,
\]
with amplitude $a=a(x,y,\xi)\in C^\infty(\Rone^{n}
\times\Rn\times\Rn)$ satisfying
\[
|\partial_x^\alpha\partial_y^\beta\partial_\xi^\gamma
a(x,y,\xi)|\leq C_{\alpha\beta\gamma}\bra{x}^{m_1}
\bra{y}^{m_2-|\beta|} \bra{\xi}^{m_3},
\]
for all $\alpha, \beta,\gamma$, and all $x,y,\xi\in\Rn$.
Then $T$ can be written in the form of a
pseudo-differential operator 
\[
Tu(x)=\int_\Rn e^{ix\cdot\xi} c(x,\xi) \widehat{u}(\xi) 
\dslash\xi,
\]
with amplitude $c(x,\xi)$ satisfying
\[
|\partial_x^\alpha\partial_\xi^\beta c(x,\xi)|\leq
C_{\alpha\beta}\bra{x}^{m_1+m_2} \bra{\xi}^{m_3},
\]
for all $\alpha,\beta$, and all $x,\xi\in\Rn$.
Moreover, we have the asymptotic expansion
\[
c(x,\xi) \thicksim \sum_\beta \frac{i^{-|\beta|}}{\beta !}
\partial_\xi^\beta\partial_y^\beta
a(x,y,\xi)|_{y=x}.
\]
\end{cor}

Note that from all asymptotic expansions it is clear that if $a$
has additional decay with respect to some variables, so does the
new amplitude (see Remarks \ref{rem1},
\ref{rem2} and \ref{rem3} for such conclusions).

\section{Weighted estimates in Sobolev spaces}

Again, we will be studying operators
\begin{equation}\label{eq:Tbasic}
Tu(x)=\int_\Rn\int_\Rn e^{i(\phi(x,\xi)-y\cdot\xi)} a(x,y,\xi)
u(y) \,dy\,\dslash\xi
\end{equation}
with real-valued phase $\phi=\phi(x,\xi)\in
C^\infty(\Rn\times\Rn)$ and amplitude $a=a(x,y,\xi)\in
C^\infty(\Rn\times\Rn\times\Rn).$ In \cite{RS4}, we proved
weighted $L^2$ estimates for such operators, also
giving explicit estimates on the required number of
derivatives of phase and amplitude. In Theorem
\ref{th:sob1} we will give weighted Sobolev estimates
for these operators.

We will say that $f\in H^{s_1,s_2}(\Rn)$ if
$f\in\Scal^\prime(\Rn)$ and $\Pi_{s_1,s_2} f\in L^2(\Rn)$, where
$\Pi_{s_1,s_2}$ is a pseudo-differential operator with symbol
$\pi_{s_1,s_2}(x,\xi)=\bra{x}^{s_1}\bra{\xi}^{s_2}.$ We can also
define operators $\widetilde{\Pi}_{s_1,s_2}$ with symbols
$\widetilde{\pi}_{s_1,s_2}(y,\xi)=
\bra{y}^{s_1}\bra{\xi}^{s_2}.$ Then
it can be shown (e.g. \cite{Co2}) that
$\widetilde{\Pi}_{-s_1,-s_2}=\Pi^{-1}_{s_1,s_2}.$ Also, Sobolev space
defined by $\widetilde{\Pi}_{s_1,s_2}$, i.e. the space of all $f\in
\Scal^\prime(\Rn)$ with $\widetilde{\Pi}_{s_1,s_2} f\in L^2(\Rn)$,
coincides with $H^{s_1,s_2}(\Rn).$

The strategy for obtaining estimates in weighted Sobolev
spaces will be to use the calculus developed in the previous
section to reduce the problem to global $L^2$ estimates
for model cases. We will first formulate global $L^2$ estimates
for operators of the form \eqref{eq:Tbasic}. In fact, 
under certain conditions, we
will reduce them to operators of a simplified form,
for which we have the following result, which appeared as
a special case of Theorem 2.5 in \cite{RS4}:

\begin{theorem}\label{Th2.5}
Let operator $T$ be defined by
\begin{equation*}\label{T2}
 Tu(x)
 =\int_{\R^n}\int_{\R^n} e^{i(x\cdot\xi+\varphi(y,\xi))}a(y,\xi)u(y) dy d\xi.
\end{equation*}
Let the real-valued phase
$\varphi=\varphi(y,\xi)\in C^\infty\p{\R^n\times\R^n}$ satisfy
\begin{equation}\label{nondeg}
\abs{\det\partial_y\partial_\xi\varphi(y,\xi)}\geq C>0,
\end{equation}
and
\[
\abs{\partial_y^\alpha\partial_\xi \varphi(y,\xi)}\leq C_{\alpha},
\qquad
\abs{\partial_y\partial_\xi^\beta \varphi(y,\xi)}\leq C_{\beta}
\]
for all $1\leq |\alpha|,|\beta|\leq 2n+2$ and all $y,\xi\in\R^n$.
Let $a=a(y,\xi)\in C^\infty\p{\R^n\times\R^n}$ satisfy
\[
\abs{\partial_y^\alpha\partial_\xi^\beta a(y,\xi)}
\leq C_{\alpha\beta},
\]
for all $|\alpha|, |\beta|\leq 2n+1$, and all $y,\xi\in\R^n$.
Then the operator $T$ is $L^2(\R^n)$-bounded, and satisfies
\[
\n{T}_{L^2\to L^2}\leq C
\sup_{|\alpha|,|\beta|\leq 2n+1}
 \n{\partial_y^\alpha\partial_\xi^\beta a(y,\xi)}
  _{L^\infty\p{\R^n\times\R^n}}.
\]
\end{theorem}

\begin{rem}\label{rem3.1}
In Theorem 2.5 in \cite{RS4}, we treated the case that
conditions on the phase $\varphi$ are imposed only on 
a subset of $\R^n\times\R^n$,
which contains the support of $a$. 
In such case, however, condition \eqref{nondeg}
must be replaced by
\begin{equation}\label{nondeg1}
|\partial_\xi\varphi(x,\xi)-\partial_\xi
\varphi(y,\xi)|\geq C|x-y|,\quad
|\partial_y\varphi(y,\xi)-\partial_y
\varphi(y,\eta)|\geq C|\xi-\eta|.
\end{equation}
We remark that if condition \eqref{nondeg} is satisfied
for all $y,\eta\in\R^n$ as in the case of Theorem \ref{Th2.5},
then the global inverses of
$\nabla_\xi\varphi(\cdot,\xi)$ and $\nabla_y\varphi(y,\cdot)$
exist, and condition \eqref{nondeg1} is implied by the mean value
theorem applied to them.
\end{rem}

Now, we can use this result to treat operators of the
form \eqref{eq:Tbasic}:

\begin{theorem}\label{th:l2}
Let $T$ be defined by 
\begin{equation*}\label{eq:Tbasic-a}
Tu(x)=\int_\Rn\int_\Rn e^{i(\phi(x,\xi)-y\cdot\xi)} a(x,y,\xi)
u(y) \,dy\,\dslash\xi.
\end{equation*} 
Let the real-valued phase
$\phi=\phi(x,\xi)\in C^\infty(\Rn\times\Rn)$ for some positive
constants and all $|\alpha|, |\beta|\geq 1$ satisfy
\begin{equation}\label{eq:phiass1}
\begin{aligned}
& |\det\partial_x\partial_\xi\phi(x,\xi)|\geq C_0>0, \\
& \exists x_\beta: |\partial^\beta_\xi\phi(x_\beta,\xi)|\leq C_\beta, \\
& |\partial_x^\alpha\partial_\xi^\beta\phi(x,\xi)|\leq
C_{\alpha\beta},
\end{aligned}
\end{equation}
for all $x,\xi\in\Rn$. 
Let amplitude $a=a(x,y,\xi)\in C^\infty(\Rone^{n}\times\Rn
\times\Rn)$ for some
$m\in\Rone$ satisfy
\[
|\partial_x^\alpha\partial_y^\beta\partial_\xi^\gamma
a(x,y,\xi)|\leq C_{\alpha\beta\gamma}\bra{x}^{m}
\bra{y}^{-m-|\beta|}, 
\textrm{ for all }\alpha,\beta,\gamma,
\textrm{ and all }
x,y,\xi\in\Rn.
\]
Then $T$ is bounded from $L^2(\Rn)$ to $L^2(\Rn)$.
\end{theorem}

\begin{rem}\label{rem3.2}
There is an alternative formulation of Theorem
\ref{th:l2} if we impose conditions (\ref{eq:phiass1}) 
on the support of $a$ only, similar
to Theorem 2.5 in \cite{RS4}
(see Remark \ref{rem3.1}).
However, in this
case we also need to assume that
\[
\bra{x}\leq C\bra{\nabla_\xi\phi(x,\xi)}
\]
on the support of $a$ as well,
since it does not automatically follow from conditions
(\ref{eq:phiass1}) .
\end{rem}

\begin{proof}[Proof of Theorem \ref{th:l2}]
By Corollary \ref{cor:repr} we can reduce $T$ to
\[
S u(x)=\int_\Rn e^{i\phi(x,\xi)}c(x,\xi) \widehat{u}(\xi) 
\dslash\xi,
\]
with amplitude satisfying $|\partial_x^\alpha\partial_\xi^\beta
c(x,\xi)|\leq C_{\alpha\beta},$ for all $\alpha,\beta$. We only
need to check that conditions (\ref{eq:phase21}) of Theorem
\ref{th:tp} on the phase are fulfilled. For $x_\beta$ 
as in \eqref{eq:phiass1}, and some $z$
we have
\begin{equation*}\label{eq:phrepr}
\partial^\beta_\xi\phi(x,\xi)-\partial^\beta_\xi\phi(x_\beta,\xi)=
\nabla_x\partial^\beta_\xi \phi(z,\xi)(x-x_\beta).
\end{equation*}
It follows from (\ref{eq:phiass1}) that
\[
\bra{\partial^\beta_\xi\phi(x,\xi)}\leq
1+|\partial_\xi^\beta\phi(x,\xi)|\leq 1+C|x-x_\beta|+C_\beta\leq
C\bra{x}.
\]
So assumption (\ref{eq:phase21}) is fulfilled if we show that
$\bra{x}\leq C_1\bra{\nabla_\xi\phi(x,\xi)}.$ 
From the first inequality in \eqref{eq:phiass1}, we obtain
inequality
$$
|\nabla_\xi\varphi(x,\xi)-\nabla_\xi
\varphi(x_\beta,\xi)|\geq C|x-x_\beta|,
\text{ for all } x,\xi\in\Rn
$$
(see Remark \ref{rem3.1}).
Now estimate $\bra{x}\leq C_1\bra{\nabla_\xi\phi(x,\xi)}$
follows from the second inequality in \eqref{eq:phiass1}. 
The adjoint $S^*$ of $S$ satisfies conditions of 
Theorem \ref{Th2.5}, so it is bounded in
$L^2(\Rn).$
\end{proof}

Now we will apply this as well as the calculus developed in
the previous section to obtain
estimates in weighted Sobolev spaces.

\begin{theorem}\label{th:sob1}
Let operator $T$ be defined by 
\begin{equation*}\label{eq:Tbasic-w}
Tu(x)=\int_\Rn\int_\Rn e^{i(\phi(x,\xi)-y\cdot\xi)} a(x,y,\xi)
u(y) \,dy\,\dslash\xi.
\end{equation*}
Let the real valued phase
$\phi=\phi(x,\xi)\in C^\infty(\Rn\times\Rn)$ for all $|\alpha|,
|\beta| \geq 1$ and all $x,\xi\in\Rn$ satisfy
\begin{equation} \label{eq:phi-sob}
|\det\partial_x\partial_\xi\phi(x,\xi)|\geq C_0>0, \quad
|\partial_x^\alpha\phi(x,\xi)|\leq C_\alpha\bra{\xi}, \quad
|\partial_x^\alpha\partial_\xi^\beta\phi(x,\xi)|\leq
C_{\alpha\beta}.
\end{equation}
Assume one of the following:

\begin{itemize}
\item[(1)] For all $\alpha$, $\beta$, and $\gamma$, assume
\[
 |\partial_x^\alpha\partial_y^\beta\partial_\xi^\gamma
 a(x,y,\xi)|\leq C_{\alpha\beta\gamma}
 \bra{x}^{m_1}\bra{y}^{m_2-|\beta|}\bra{\xi}^{m_3},
\]
and for all $|\beta|\geq 1$, assume
\[
\exists x_\beta: |\partial^\beta_\xi\phi(x_\beta,\xi)|\leq C_\beta
\]
for all $x,y,\xi\in\Rn$.

\item[(2)] For all $\alpha$, $\beta$, and $\gamma$, assume
\[
 |\partial_x^\alpha\partial_y^\beta\partial_\xi^\gamma
 a(x,y,\xi)|\leq C_{\alpha\beta\gamma}
 \bra{x}^{m_1-|\alpha|}\bra{y}^{m_2}\bra{\xi}^{m_3},
\]
and for all $\alpha$ and $|\beta|\geq 1$, assume
\[
|\partial_x^\alpha\phi(x,\xi)|\leq C_{\alpha}
\jp{x}^{1-|\alpha|}\jp{\xi},\quad
 |\partial_x^\alpha\partial_\xi^\beta\phi(x,\xi)|\leq
 C_{\alpha\beta}\bra{x}^{1-|\alpha|},
\]
for all $x,y,\xi\in\Rn$.
\end{itemize}
Then $T$ is bounded from $H^{s_1,s_2}(\Rn)$ to $H^{s_1-m_1-m_2,
s_2-m_3}(\Rn)$, for all $s_1, s_2\in\Rn.$
\end{theorem}
\begin{proof}
Let us first note that for $f\in H^{s_1,s_2}$ we have
\[
\begin{aligned}
|| Tf||_{H^{s_1-m_1-m_2,s_2-m_3}} & \leq
||\Pi_{s_1-m_1-m_2,s_2-m_3} Tf||_{L^2} \\
& \leq ||\Pi_{s_1-m_1-m_2,s_2-m_3} T
\widetilde{\Pi}_{-s_1,-s_2}||_{L^2\to L^2} ||\Pi_{s_1,s_2} f||_{L^2}
\\
& \leq C||f||_{H^{s_1,s_2}},
\end{aligned}
\]
if we show that $\Pi_{s_1-m_1-m_2,s_2-m_3} T
\widetilde{\Pi}_{-s_1,-s_2}$ is bounded in $L^2(\Rn).$
{Assume (1)}. 
If we apply Theorems \ref{prop:alt} and \ref{th:pt} 
(and also the argument of Remarks \ref{rem1} and \ref{rem2}) to
the composition
$\jp{D}^{s_2-m_3} T \jp{D}^{-s_2}$, we can conclude that
$S=\Pi_{s_1-m_1-m_2,s_2-m_3} T \widetilde{\Pi}_{-s_1,-s_2}$ is an
operator of the form (\ref{eq:Tbasic}) with amplitude $c(x,y,\xi)$
satisfying
\begin{equation}\label{EQ:weighted-amp}
|\partial_x^\alpha\partial_y^\beta\partial_\xi^\gamma c(x,y,\xi)|
\leq C_{\alpha\beta\gamma} \bra{x}^{s_1-m_2} \bra{y}^{m_2-s_1-|\beta|},
\end{equation}
for all $\alpha\,\beta,\gamma$, and for all $x,y,\xi\in\Rn$.
Here we have used the fact that conditions (\ref{eq:phi-sob})
guarantee conditions on the phase function in Theorems
\ref{prop:alt} and \ref{th:pt} (similar to the proof of Theorem \ref{th:l2}).
Now, {in the case (1)}, by
Theorem \ref{th:l2} we obtain that $S$ is bounded in $L^2(\Rn).$

Now assume (2). By Remark \ref{rem3} we get that
the amplitude of the operator $S$ above is again of the form
(\ref{eq:Tbasic}) and it satisfies estimate 
\begin{equation*}\label{EQ:weighted-amp2}
|\partial_x^\alpha\partial_y^\beta\partial_\xi^\gamma c(x,y,\xi)|
\leq C_{\alpha\beta\gamma} \bra{x}^{s_1-m_2-|\alpha|} 
\bra{y}^{m_2-s_1},
\end{equation*}
again implying the $L^2(\Rn)$--boundedness
of $S$ by Theorem 3.1 from \cite{RS4}.

\end{proof}

\section{Global smoothing for hyperbolic equations}
In this section we will consider the following Cauchy problem
\begin{equation}
\label{eq:sm}
\left\{
\begin{array}{l}
i\partial_t u -a(D_x) u = 0, \\
u(0,x)=u_0(x).
\end{array}
\right.
\end{equation}
The estimate of Theorem \ref{ths} of this
section can be also viewed as a quantitative 
way to express the fact that 
space decay of Cauchy data $u_0$ translates into the time decay
of the solution $u$. 

Let us assume that operator $a(D_x)$ has a real valued smooth 
symbol
$a=a(\xi)\in C^\infty(\R^n)$ satisfying 
$$\nabla a(\xi)\neq0 \textrm{ for all }
\xi\in\Rn.$$ 
This condition makes the equation
``dispersive'', ensuring that the singularities will escape
to infinity as $t\to\pm\infty$. The dispersiveness is in
general necessary to have the smoothing properties for
evolution equations,
although some smoothing still takes place without this
assumption (see \cite{RS6}).

Assume also that
\[
 a(\xi)=a_1(\xi)+a_0(\xi) \quad\textrm{ for large } \xi,
\]
where $a_1=a_1(\xi)\in C^\infty(\R^n\setminus0)$ 
is a positively homogeneous function
of order $1$ (i.e. $a_1(\lambda\xi)=\lambda a_1(\xi)$ for
all $\lambda>0$ and $\xi\not=0$)
satisfying $\nabla a_1(\xi)\neq0$ for all $\xi\neq0$,
and $a_0=a_0(\xi)\in S^0$ is a symbol of order zero,
i.e. a smooth function such that for every multi-index 
$\alpha$ there exists a constant $C_\alpha>0$ such that
$$|\partial^\alpha a_0(\xi)|\leq C_\alpha\jp{\xi}^{-|\alpha|}
\quad \textrm{ for all } \xi\in\Rn.$$
We note that from these assumptions it follows that $a$ itself
is a symbol of order one.

In the case $n=1$, we can easily see that
the following estimate for the solution
$u=u(t,x)$ of \eqref{eq:sm} is true for $s>1/2$:
\begin{equation}
\label{e1}
||\jp{x}^{-s}u||_{L^2(\Rone_t\times\R^n_x)}\leq C 
||u_0||_{L^2(\R^n_x)}.
\end{equation}
Note that if $a(\xi)$ itself is homogeneous,
then $u(t,x)$ is just a translation of $u_0$,
so \eqref{e1} trivially holds.
In \cite{RS2} we have presented a result 
(which in turn follows from more general results proved in 
\cite{RS6}) that implies
that estimate \eqref{e1} is also true for all $n\ge1$.
Analysis of previous sections allows us to
apply this to obtain further global properties 
of solutions to \eqref{eq:sm}. 

\begin{theorem}
Let $u(t,x)$
be the solution of the Cauchy problem \eqref{eq:sm}.
Then for $s>1/2$ and for all $k=0,1,2,\ldots$ we have estimates
\begin{equation}
\label{tk}
||t^k\jp{x}^{-k}\jp{x}^{-s} u||_{L^2(\Rone_t\times\Rnx)}
\leq C_k ||\jp{x}^k u_0||_{L^2(\Rnx)},
\end{equation}
where constant $C_k$ is independent of Cauchy data $u_0$.
\label{ths}
\end{theorem}
\begin{proof}
For $k=0$ formula \eqref{tk} is the same as \eqref{e1}.
Let now $k\geq 1$. We want to show that estimate
\begin{equation}
\label{tka}
||t^k\jp{x}^{-s_k} u||_{L^2(\Rone_t\times\Rnx)}
\leq C_k ||\jp{x}^k u_0||_{L^2(\Rnx)},
\end{equation}
holds for $s_k>k+\frac12$.

The solution $u=u(t,x)$ to \eqref{eq:sm} is given by
$$ u(t,x)=T_tu_0(x)=\int_\Rn\int_\Rn e^{i[(x-y)\cdot\xi+t a(\xi)]} u_0(y)
dy \dslash\xi.$$
Let $\sigma(X,D)$ be a pseudo-differential operator with
symbol $\sigma(x,\xi)=x\cdot\nabla a(\xi)$.
Let us denote
 $$\tau(t,x,\xi):=\sigma(x+t\nabla a(\xi),\xi)=
(x+t\nabla a(\xi))\cdot\nabla a(\xi).$$
Taking $\phi(x,\xi)=x\cdot\xi+ta(\xi)$,
it is straightforward to see that
function $\Psi$ defined in Theorem \ref{th:tp} is
$$\Psi(\eta,\xi,x)=x\cdot(\xi-\eta)+t(a(\xi)-a(\eta))+(\eta-\xi)
\cdot(x+t\nabla a(\xi)).$$ 
We can observe that 
$\nabla_{\eta}e^{i\Psi(\eta,\xi,x)}|_{\eta=\xi}=0.$
We can also observe that the corresponding function $\Psi(x,y,\xi)$
in Theorem \ref{th:pt} is identically zero.
Since $\sigma$ is affine in $x$, it follows from Theorems
\ref{th:pt} and \ref{th:tp} that
$$\tau(t,X,D)\circ T_t=T_t\circ\sigma(X,D).$$
Now, taking $\sigma(X,D)u_0$ as the Cauchy data in \eqref{eq:sm}
and using \eqref{e1} with $s>1/2$, we get
\begin{equation*}
\label{e2}
||\jp{x}^{-s}T_t\circ\sigma(X,D) u_0||_{L^2(\Rone_t\times\Rnx)}
\leq C ||\sigma(X,D) u_0||_{L^2(\Rnx)}.
\end{equation*}
This means that we have the estimate
\begin{equation*}
\label{e3}
||\jp{x}^{-s}(x+t\nabla a(D))\cdot\nabla a(D)u||_
{L^2(\Rone_t\times\Rnx)}
\leq C ||x\cdot\nabla a(D) u_0||_{L^2(\Rnx)}
\end{equation*}
for $s>1/2$.
Since $\nabla a(D)$ 
is a vector of pseudo-differential operators of order zero,
we have the estimate 
$$||x\cdot\nabla a(D)u_0||_{L^2(\Rnx)}\leq
C||\jp{x}u_0||_{L^2(\Rnx)}$$
using the
$L^2(\Rn)$-boundedness of the operator 
$x\cdot\nabla a(D)\jp{x}^{-1}$
by Theorem \ref{th:sob1}.
Applying the same argument $k$ times, we get estimates
\begin{equation}
\label{e3l}
||\jp{x}^{-s}\left[(x+t\nabla a(D))\cdot\nabla a(D)
\right]^l u||_
{L^2(\Rone_t\times\Rnx)}
\leq C ||\bra{x}^l u_0||_{L^2(\Rnx)},
\end{equation}
for all $l=1,\ldots,k$ and $s>1/2$.
Now, it follows that
\begin{equation}\label{esta1}
\begin{aligned}
& ||t^k\,\jp{x}^{-s_k} u||_
{L^2(\R_t\times\Rnx)} \\
& =
||\jp{x}^{-s_k}|\nabla a(D)|^{-2k}
\left[(x+t\nabla a(D))\cdot\nabla a(D)
 -x\cdot\nabla a(D)\right]^k u||_{L^2(\R_t\times\Rnx)} \\
 & \leq C||\jp{x}^{-s_k}
\left[(x+t\nabla a(D))\cdot\nabla a(D)
 -x\cdot\nabla a(D)\right]^k u||_{L^2(\R_t\times\Rnx)}, 
\end{aligned}  
\end{equation}
where we have eliminated 
$|\nabla a(D)|^{-2k}$ from the estimate using the
$L^2(\Rn)$-boundedness of the operator 
$\jp{x}^{-s_k}|\nabla a(D)|^{-2k}\jp{x}^{s_k}$
by Theorem \ref{th:sob1}.
Now, we can estimate
\begin{equation}\label{esta2}
 \begin{aligned}
 & ||\jp{x}^{-s_k}
\left[(x+t\nabla a(D))\cdot\nabla a(D)
 -x\cdot\nabla a(D)\right]^k u||_{L^2(\R_t\times\Rnx)} \\  
 & \leq
 \sum_{l=0}^k C_k^l
 ||\jp{x}^{-s_k}
 \left[x\cdot\nabla a(D)\right]^{k-l}
 \left[(x+t\nabla a(D))\cdot\nabla a(D)\right]^l
  u||_{L^2(\R_t\times\Rnx)} \\ 
 & \leq
 \sum_{l=0}^k C_k^l
 ||\jp{x}^{-s_k+k-l}
  \left[(x+t\nabla a(D))\cdot\nabla a(D)\right]^l
  u||_{L^2(\R_t\times\Rnx)},
 \end{aligned}
\end{equation}  
where once again we used
the weighted $L^2(\Rn)$--boundedness of
$|\nabla a(D)|^{k-l}$ to eliminate it from the last inequality.
We also note that operators in \eqref{esta2} are of the same
form, or are scalar, and they satisfy weighted 
$L^2(\Rn)$--boundedness,
so they have good commutator properties as well.
Since $$s_k-k+l>k+\frac12-k+l\geq \frac12,$$ we can use
\eqref{e3l}. Combining it with \eqref{esta1} and
\eqref{esta2}, we conclude that
$$  
  ||t^k\,\jp{x}^{-s_k} u||_
{L^2(\R_t\times\Rnx)} \leq
C||\jp{x}^k u_0||_{L^2(\Rnx)},
$$
completing the proof of \eqref{tk}. 
\end{proof}
We remark that estimate \eqref{tk} 
with $k=1$ is sharp in the sense that
in general it does not hold for $s=1/2$. 
For example, its failure can 
be easily seen for $n=1$ and $a(\xi)=\xi$.

We have the following version of Theorem \ref{ths} if 
$a(\xi)$ itself is positively homogeneous of order one.
In this case $a(\xi)$ has a singularity at the origin 
which causes a small additional complication when we
eliminate operators of the form
$|\nabla a(D)|^{-2k}$ from weighted estimates.
In this case we have the following corollary from the
proof of Theorem \ref{ths}.

\begin{cor}
If $a(\xi)$ is positively homogeneous of order one with 
$\nabla a(\xi)\not=0$ for all $\xi\not=0$,
estimate \eqref{tk} of
Theorem \ref{ths} still holds for $0\leq k< \frac{n-1}{2}$.
\end{cor}
To justify it, we note first that operator
$|\nabla a(D)|$
has symbol which is positively homogeneous of order zero.
Consequently, for any $l$,
operators $\jp{x}^{-s_k}|\nabla a(D)|^{l}\jp{x}^{s_k}$ is 
$L^2(\Rn)$-bounded
for $-\frac{n}{2}< s_k< \frac{n}{2}$
(see Kurtz and Wheeden \cite{KW}, for example). Since we also need 
$s_k>k+\frac12$, we see that the conclusion of Theorem
\ref{ths} remains valid for $k<\frac{n-1}{2}.$


\par

\end{document}